\documentclass{amsart}
\usepackage{amssymb}
\usepackage{amsfonts}
\usepackage{amssymb}
\usepackage{amsmath, accents}
\usepackage{amsthm}
\usepackage{enumerate}
\usepackage[nolabel, final]{showlabels}
\usepackage{tabularx}
\usepackage{centernot}
\usepackage{mathtools}
\usepackage{stmaryrd}
\usepackage{amsthm,amssymb}
\usepackage{etoolbox}
\usepackage{tikz}
\usepackage{amssymb}
\usetikzlibrary{matrix}
\usepackage{tikz-cd}
\usepackage{tikz}
\usepackage{marginnote}
\definecolor{mygray}{gray}{0.85}
\usepackage[backgroundcolor=mygray,colorinlistoftodos,prependcaption,textsize=small]{todonotes}
\usepackage{xcolor}
\usepackage{showlabels}

\renewcommand{\leq}{\leqslant}
\renewcommand{\geq}{\geqslant}

\makeatletter
\def\subsection{\@startsection{subsection}{3}%
  \z@{.5\linespacing\@plus.7\linespacing}{.3\linespacing}%
  {\bfseries\centering}}
\makeatother

\makeatletter
\def\subsubsection{\@startsection{subsubsection}{3}%
  \z@{.5\linespacing\@plus.7\linespacing}{.3\linespacing}%
  {\centering}}
\makeatother

\makeatletter
\def\myfnt{\ifx\protect\@typeset@protect\expandafter\footnote\else\expandafter\@gobble\fi}
\makeatother

\newtheorem{theorem}{Theorem}[section]

\newtheorem{corollary}[theorem]{Corollary}
\newtheorem{definition}[theorem]{Definition}
\newtheorem{lemma}[theorem]{Lemma}

\newtheorem{question}[theorem]{Question}
\newtheorem{observation}[theorem]{Observation}
\newtheorem{fact}[theorem]{Fact}

\newtheorem{remark}[theorem]{Remark}

\newtheorem{choice}[theorem]{Choice}
\newtheorem{notation}[theorem]{Notation}

\newtheorem{hypothesis}[theorem]{Hypothesis}

\newtheorem{cclaim}[theorem]{Claim}

\newtheorem{convention}[theorem]{Convention}

\newtheorem*{proviso}{Proviso}

\newcounter{claimcounter}
\numberwithin{claimcounter}{theorem}

\usepackage{eso-pic}
\usepackage{blindtext} 

\newcommand{\mrm}[1]{\mathrm{#1}}

\usepackage{pdfpages}
\usepackage{hyperref}

\begin{document}

\begin{abstract} Relying on the techniques and ideas from our recent paper \cite{1205}, we prove several anti-classification results for various rigidity conditions in countable abelian and nilpotent groups. We prove three main theorems: (1) the rigid abelian groups are complete co-analytic in the space of countable torsion-free abelian groups ($\mrm{TFAB}_\omega$); (2) the Hopfian groups are complete co-analytic in $\mrm{TFAB}_\omega$; (3) the co-Hopfian groups are complete co-analytic in the space of countable $2$-nilpotent groups. In combination with our result from \cite[S5]{1205}, which shows that the endo-rigid abelian groups are complete co-analytic in $\mrm{TFAB}_\omega$, this shows that four major notions of rigidity from (abelian) group theory are as complex as possible as co-analytic problems. Further, the second and third theorem above solve two open questions of Thomas from \cite{thomas_complete_groups}, who asked this for the space of all countable groups. We leave open the question of whether the co-Hopfian mixed abelian groups are complete co-analytic in the space of countable abelian groups, but we reduce the problem to a concrete question on profinite groups, showing that if $G$ is a countable co-Hopfian abelian reduced group, then, for every prime number $p$, the torsion subgroup $\mrm{Tor}_p(G)$ of $G$ is finite and $G$ embeds in the profinite group $ \prod_{p \in \mathbb{P}} \mrm{Tor}_p(G)$.
\end{abstract}

\title[Anti-Classification Results for Rigid Abelian and Nilpotent Groups]{Anti-Classification Results for Rigidity Conditions in Abelian and Nilpotent Groups}


\thanks{No. 1237 on Shelah's publication list. Research of the first author was partially supported by project PRIN 2017 ``Mathematical Logic: models, sets, computability", prot. 2017NWTM8R and project PRIN 2022 ``Models, sets and classifications", prot. 2022TECZJA. Research of the second author was partially supported by the grant ``Independent Theories'' NSF-BSF, (BSF 3013005232) and by Israel Science Foundation (ISF) grants no: 1838/19 and 2320/23.}

\author{Gianluca Paolini}
\address{Department of Mathematics ``Giuseppe Peano'', University of Torino, Via Carlo Alberto 10, 10123, Italy.}

\author{Saharon Shelah}
\address{Einstein Institute of Mathematics,  The Hebrew University of Jerusalem, Israel \and Department of Mathematics,  Rutgers University, U.S.A.}

\date{\today}
\maketitle




\section{Introduction}

	The set of all groups with domain $\omega$ ($\omega$ is just another notation for $\mathbb{N}$) forms naturally a Borel space $\mathrm{Gp}_{\omega}$. This allows us to formalize classification problems in countable group theory in terms of complexity in the Borel space $\mathrm{Gp}_{\omega}$, using the usual notions of complexity for subsets of $\mathbb{R}$, such as analytic sets, co-analytic sets, etc. This makes sense not only for groups but actually for any first-order\footnote{In fact for any theory $T$ in the logic $\mathfrak{L}_{\omega_1, \omega}$, but assuming $T$ is first-order suffices here.} theory $T$ and it is the premise of what is known as {\em invariant descriptive set theory of countable structures}\footnote{For an extensive introduction to this area of research see the classical reference \cite{gao}, or for a quick introduction to the notions used in this paper see the introduction of our paper \cite{1205}.}. In this context, we recently proved in \cite{1205}\footnote{As of 04.12.2023, \cite{1205} has been accepted for publication in Ann. of Math. (2), link \underline{\href{https://annals.math.princeton.edu/articles/21268}{here}}.} that the Borel space of countable torsion-free abelian groups ($\mrm{TFAB}_\omega$) is as complex as possible in terms of classification up to isomorphism, resolving an important conjecture of Friedman and Stanley from '89 (cf. \cite{friedman_and_stanley}). In combination with fundamental results of Thomas \cite{thomas_ACTA, thomas_JAMS} on $\mrm{TFAB}_\omega$ of finite rank, this gives a complete solution to the isomorphism problem for $\mrm{TFAB}_\omega$. Most certainly, this is not the end of the story concerning interactions between descriptive set theory of countable structures and classification problems in countable abelian group theory. On one hand, we can consider classification problems up to other equivalence relations of interest (see e.g. the  work of Calderoni and Thomas \cite{calderoni} on classification of countable abelian groups up to bi-embeddability), and on the other hand we can consider classification problems for {\em algebraic properties} of (abelian) groups (where we look at properties as subsets of $\mathrm{Gp}_{\omega}$). In this respect, we proved in \cite[S5]{1205} that the endorigid (see below) groups in $\mrm{TFAB}_\omega$ form a complete co-analytic subset of $\mrm{TFAB}_\omega$, and so that the classification of these groups is as complex as possible as a co-analytic problem. In this paper we continue these investigations, relying on the elaborated tools developed in our paper \cite{1205},  focusing on other notions of rigidity which are central to (abelian) group theory. To this extent we now recall the following definitions:
	 
	\begin{definition}\label{the_rigidity_def} Let $G$ be a group.
	 \begin{enumerate}[(1)]
	 \item  We say that $G$ is Hopfian if every onto endomorphism of $G$ is $1$-to-$1$.
	 \item We say that $G$ is co-Hopfian if every $1$-to-$1$ endomorphism of $G$ is onto.
	 \item If $G$ is abelian, we say that $G$ is rigid when $\mrm{Aut}(G) = \{ \mrm{id}_G, -\mrm{id}_G \}$.
	 \item If $G$ is abelian, we say that $G$ is endorigid if $f \in \mrm{End}(G)$ implies that there is $m \in \mathbb{Z}$ s.t., for all $a \in G$, $f(a) = ma$ (i.e., $f$ is multiplication by an integer).
\end{enumerate}
\end{definition}	 

	Each of the four notions above, although arising from different contexts, are most certainly of similar nature, i.e., they are all {\em rigidity conditions}. These kind of conditions are often considered of great importance by algebraists and set-theorists (see below for some historical background). As already mentioned, in \cite[S5]{1205} we showed that the endorigid abelian groups are complete co-analytic. The aim of this paper is to show that also the three other notions of rigidity from \ref{the_rigidity_def} are complete co-analytic.  
	Further, this solves the two following problems of Thomas from \cite{thomas_complete_groups}: 
	
	\begin{question}[Thomas]\label{thomas_question}
	\begin{enumerate}[(1)]
	\item Are the Hopfian groups complete co-analytic in~$\mrm{Gp}_\omega$?
	\item Are the co-Hopfian groups complete co-analytic in $\mrm{Gp}_\omega$?
\end{enumerate}	
\end{question}

	Our solutions to the problems of Thomas assume a very strong form, i.e., we show that the Hopfian groups are already complete co-analytic in the space of torsion-free abelian groups with domain $\omega$ and that the co-Hopfian groups are already complete co-analytic in the space of $2$-nilpotent groups with domain $\omega$. Clearly, showing that such notions are complete co-analytic in the abelian or $2$-nilpotent context is a much harder task than proving it for {\em all} countable groups, as one has much less freedom for the construction \mbox{exhibiting completeness to take place.}


	\medskip

	We spend the rest of the introduction giving some context and background to our problems. We start with the (co-)Hopfian properties. The notions of  Hopfian and co-Hopfian group have been studied for a long time, under different names. In the context of abelian group theory they were first considered by Baer in \cite{baer}, where he refers to them as $Q$-groups and $S$-groups. The modern terminology arose from the work of the German mathematician H. Hopf, who showed that the defining property holds of certain two dimensional manifolds. The research on Hopfian and co-Hopfian abelian groups has recently been revived thanks to the discovered connections with the study of algebraic entropy and its dual (see \cite{entropy1, entropy2}), as e.g. groups of zero algebraic entropy are necessarily co-Hopfian (for more on the connections between these two topics see \cite{entropy3}). An easy observation shows that a torsion-free abelian group is co-Hopfian if and only if it is divisible of finite rank, hence the problem of classification of such groups naturally reduces to the torsion and mixed cases. A major progress in this line of research was made by Beaumont and Pierce in \cite{b&p} where the authors proved several general important results, in particular that if $G$ is co-Hopfian, then $\mrm{Tor}(G)$ is of size at most continuum, and further that $G$ cannot be a $p$-group of size $\aleph_0$. Since then, to the best of our knowledge, not much more has been discovered in the abelian or nilpotent context, and the main questions on co-Hopfian abelian groups moved to the uncountable case (see our paper \cite{1214} on this topic). Moving to Hopfian groups, the situation is quite different, as in this case there is no obvious obstruction for a torsion-free abelian group to be Hopfian. Furthermore, every endo-rigid abelian group is Hopfian, and several classical techniques to construct such groups are known, see e.g. \cite{sh44}. Apart from this, it seems to us that little is known on the general classification problem for torsion-free Hopfian abelian groups (a posteriori, maybe this is for a good reason, see Theorem~\ref{main_th4}). Finally, we want to mention that there are several problems on (co-)Hopfian abelian groups which remain \mbox{stubbornly open (see e.g. \cite{entropy3, hirshon} and references therein).}
	
	\medskip
	
	We continue our review of the background on our problems moving to the endo-rigid and rigid conditions (i.e., (3) and (4) from \ref{the_rigidity_def}), thus naturally assuming that $G$ is abelian. To give an idea, the canonical example of these phenomena is most certainly the group $\mathbb{Z}$, which is known to be both endo-rigid and rigid. These two notions were largely explored in the '70's and in the '80's, in particular by Fuchs, Shelah, G\"obel, Eklof, et al. In this context the interest was mostly in the uncountable case, and many constructions of such groups were exhibited over the years. The notion of rigidity received attention also in a quite different context, in fact in the fundamental work of Adams and Kechris \cite{kechris} one of the most intriguing applications of the general  ``Superrigidity Method'' was to show that the complexity of the isomorphism problem for rigid abelian groups of rank $n$ is strictly increasing with $n$ (as mentioned above, this was later improved by Thomas to all the countable torsion-free abelian groups of finite rank \cite{thomas_ACTA, thomas_JAMS}). Also in these two cases, as expected, no general classification result is known (and for a good reason).
	
\medskip

	We now introduce some notations and then conclude by stating our theorems.
	 
	\begin{notation}\label{first_notation} \begin{enumerate}[(1)]
	\item $\mathrm{AB}_\omega$ denotes the Borel space of abelian groups with domain~$\omega$.
	\item $\mathrm{TFAB}_\omega$ denotes the Borel space of torsion-free abelian groups with domain $\omega$.
	\item $\mrm{NiGp}(n)_\omega$ denotes the Borel space of nilpotent groups of class $n$ with domain~$\omega$.
	\end{enumerate}
	\end{notation}

	\begin{theorem}\label{main_th4} The rigid $G \in \mathrm{TFAB}_\omega$ are complete co-analytic in $\mathrm{TFAB}_\omega$. 
\end{theorem}

	\begin{theorem}\label{main_th4_hopfian} The Hopfian groups are complete co-analytic in $\mathrm{TFAB}_\omega$. 
\end{theorem}

\begin{theorem}\label{main_th7} The co-Hopfian groups are complete co-analytic in $\mathrm{NiGp}(2)_\omega$.
\end{theorem}

	Unfortunately, we were unable to determine if we can improve Theorem~\ref{main_th7} to the abelian case (necessarily mixed, because of what was mentioned above). In this respect, though, we were able to obtain some important additional information:
	
	\begin{theorem}\label{main_th6} If $G \in \mrm{AB}_\omega$ is co-Hopfian and reduced, then for every prime $p$ we have that $\mrm{Tor}_p(G)$ is finite and $G$ embeds in the profinite group $ \prod_{p \in \mathbb{P}} \mrm{Tor}_p(G)$.
\end{theorem}

	Notice that because of \cite[Claim~2.15]{1214} the assumption ``$G$ is reduced'' in \ref{main_th6} is without loss of generality. As mentioned, we leave the following open problem:

	\begin{question} Are the co-Hopfian groups complete co-analytic in $\mrm{AB}_\omega$?
\end{question} 

	After this paper was finished we discovered that a solution to Question~\ref{thomas_question}(1) was given in the recent preprint \cite{pinsky}, that construction uses generalized wreath products and so it does not take place in the context of abelian groups. Finally, we want to mention that in \cite{decomposability} it is proved that the indecomposable abelian groups are complete co-analytic in $\mathrm{TFAB}_\omega$, a result most certainly affine to our investigations.

\section{Notations and preliminaries}

	For the readers of various backgrounds we try to make the paper self-contained.

\subsection{General notations}

	\begin{definition}
	\begin{enumerate}[(1)]
	\item Given a set $X$ we write $Y \subseteq_\omega X$ for $Y \subseteq X$ and $|Y| < \aleph_0$.
	\item Given a set $X$ and $\bar{x}, \bar{y} \in X^{< \omega}$ we write $\bar{y} \triangleleft \bar{x}$ to mean that $\mrm{lg}(\bar{y}) < \mrm{lg}(\bar{x})$ and $\bar{x} \restriction \mrm{lg}(\bar{y}) = \bar{y}$, where $\bar{x}$ is naturally considered as a function ${\mrm{lg}(\bar{x})} \rightarrow X$.
	\item Given a partial function $f: M \rightarrow M$, we denote by $\mrm{dom}(f)$ and $\mrm{ran}(f)$ the domain and the range of $f$, respectively.
	\item For $\bar{x} \in A^n$ we write $\bar{x} \subseteq A$ to mean that $\mrm{ran}(\bar{x}) \subseteq A$, where, as usual, $\bar{x}$ is considered as a function $\{0, ..., n-1 \} \rightarrow A$.
	\item Given a set $X$ and $0 < n < \omega$ we write $\mrm{seq}_n(X)$ to denote the set of injective sequences (i.e., with no repetitions) of length $n$ of elements from $X$. 
	\end{enumerate}
\end{definition}

	\begin{definition} Let $(X, <)$ be a partial order and $(x_n : n < \omega) \in X^\omega$.
	\begin{enumerate}[(1)]
	\item We say that $(x_n : n < \omega)$ is strictly increasing when we have:
	 $$n < m < \omega \; \Rightarrow \; x_n < x_m.$$
	\item We say that $(x_n : n < \omega)$ is non-decreasing when we have:
	 $$n \leq m < \omega \; \Rightarrow \; x_n \leq x_m.$$
	\end{enumerate}
\end{definition}

\subsection{Trees}



	\begin{definition}\label{def_trees} Let $(T, <_T)$ be a strict partial order.
	\begin{enumerate}[(1)]
	\item $(T, <_T)$ is a {\em tree} when, for all $t \in T$, $\{s \in T : s <_T t\}$ is well-ordered by the relation $<_T$. Notice that according to our definition a tree $(T, <_T)$ might have more than one root, i.e. more than one $<_T$-minimal element. We say that the tree $(T, <_T)$ is rooted when it has only one $<_T$-minimal element (its root).
	\item A {\em branch} of the tree $(T, <_T)$ is a maximal chain of the partial order $(T, <_T)$.
	\item A tree $(T, <_T)$ is said to be {\em well-founded} if it has only finite branches.
	\item Given a tree $(T, <_T)$ and $t \in T$ we let the level of $t$ in $(T, <_T)$, denoted as $\mrm{lev}(t)$, to be the order type of $\{s \in T : s <_T t\}$ (recall item (1)).
	\end{enumerate}
\end{definition}

	Concerning Def.~\ref{def_trees}(4), we will only consider trees $(T, <_T)$ such that, for every $t \in T$, $\{s \in T : s <_T t\}$ is finite, so for us $\mrm{lev}(t)$ will always be a natural number.
	
	\begin{notation} We denote by $\mathbb{P}$ the set of primes.
\end{notation}

\subsection{Groups}

\begin{notation} Let $G$ and $H$ be groups.
\begin{enumerate}[(1)]
	\item $H \leq G$ means that $H$ is a subgroup of $G$.
	\item We let $G^+ = G \setminus \{ e_G \}$, where $e_G$ is the neutral element of $G$.
	\item If $G$ is abelian we might denote the neutral element $e_G$ simply as $0_G = 0$.
\end{enumerate}
\end{notation}

	\begin{definition}\label{def_pure} Let $H \leq G$ be groups, we say that $H$ is pure in $G$, denoted by $H \leq_* G$, when if $k \in H$, $n < \omega$ and (in additive notation) $G \models ng = k$, then there is $h \in H$ such that $H \models nh = g$.
\end{definition}

	\begin{observation}\label{obs_pure_TFAB} If $H \leq_* G \in \mrm{TFAB}$, $k \in H$ and $0 < n < \omega$, then:
$$G \models ng = k \Rightarrow g \in H.$$
\end{observation}

	\begin{observation}\label{generalG1p_remark} Let $G \in \mathrm{TFAB}$ and let:
	$$G_{(1, p)} = \{ a \in G_1 : a \text{ is divisible by $p^m$, for every $0 < m < \omega$}\},$$
then $G_{(1, p)}$ is a pure subgroup of $G_1$.
\end{observation}

	\begin{proof} This is well-known.
\end{proof}

	\begin{definition}\label{def_basis} Let $G \in \mrm{TFAB}$, we say that $X \subseteq G$ is a basis of $G$ when $X$ is a maximal independent subset of $G$ (in the sense of \cite[Section~16]{fuch_vol1}).
\end{definition}

	\begin{definition}\label{def_free_auto} Let $A \in \mathrm{TFAB}$ and $\pi \in \mrm{Aut}(A)$. We say that $\pi$ is free if there exists a basis (cf. Def.~\ref{def_basis}) $X$ of $A$ such that $\pi$ induces a permutation of $X$.
\end{definition}

	\begin{fact}[{\cite[Lemma~1.13]{nilpotentn_book}}]\label{2nilpotent_fact} Let $G$ be nilpotent of class $2$. For every $g, h \in G$ and $m, n \in \mathbb{Z}$ we have that the following equation is true:
	$$[g^n, h^m] = [g, h]^{nm}.$$
\end{fact}

\section{The rigidity problem for countable $\mrm{TFAB}$}\label{sec_cohop}

	\begin{hypothesis}\label{hyp_A1} Throughout this section the following hypothesis stands:
	\begin{enumerate}[(1)]
	\item\label{omega_levels} $T = (T, <_T)$ is a countable rooted tree with $\omega$ levels;
	\item $T = \bigcup_{n < \omega} T_n$, $T_n \subsetneq T_{n+1}$, and $t \in T_n$ implies that $\mathrm{lev}(t) \leq n$;
	\item\label{hyp_A1_item3} $T_0 = \{t_0 \}$, where $t_0$ is the root of $T$, and, for $n < \omega$, $T_n$ is finite;
	\item $T_{<n} = \bigcup_{\ell < n} T_\ell$ (so $T_{<(n+1)} = T_n$ and $T_{<0} = \emptyset$);
	\item if $s <_T t \in T_n$, then $s \in T_{< n}$;
	\item for $t \in T$, we let $\mathbf{n}(t)$ be the unique $0 \leq n < \omega$ such that $t \in T_n \setminus \bigcup_{\ell < n} T_{\ell}$.
	\end{enumerate}
\end{hypothesis}

	\begin{definition}\label{hyp_A2} Let $\mathrm{K}^{\mrm{ri}}(T)$ be the class of objects:
	$$\mathfrak{m}(T) = \mathfrak{m} = (X^T_n, \bar{f}^T: n < \omega) = (X_n, \bar{f} : n < \omega)$$
satisfying the following requirements:
	\begin{enumerate}[(a)]
	\item\label{X0_not_empty} $X_0 \neq \emptyset$, $X_n \subseteq \omega$ is finite and strictly increasing with $n$, and $X_{< n} = \bigcup_{\ell < n} X_\ell$;
	\item we let $X = \bigcup_{n < \omega} X_n$;
	\item $\bar{f} = (f_t : t \in T)$;
	\item\label{def_of_f_t} if $0 < n < \omega$ and $t \in T_n \setminus T_{< n}$, then $f_t$ is a one-to-one function such that $X_{n-1} = \mrm{dom}(f_t) \cap \mrm{ran}(f_t)$ and $\mrm{dom}(f_t) \cup \mrm{ran}(f_t) \subseteq X_n$;
	\item for $i \in \{1, -1\}$ and $t \in T$, we let $f^i_t = f_t$, if $i = 1$, and $f^i_t = f^{-1}_t$, if $i = -1$;
	\item\label{item_fs_sub_ft} if $s <_T t \in T_n$, then $f_s \subseteq f_t$; 
	\item\label{item_no_fix_point} 
	\begin{enumerate}[$(\alpha)$]
	\item if $t$ is the root of $T$, then $f_t$ is the empty function;
	\end{enumerate}
	\begin{enumerate}[$(\beta)$]
	\item if $t \in T$ is not the root of $T$, $s <_T t$ is $<_T$-maximal below $t$, $t \in T_n \setminus T_{<n}$, $f_t(x) = y$ and $x \notin \mrm{dom}(f_s)$, then either of the following two happens:
	$$(x \in X_{n-1} \setminus \mrm{dom}(f_s) \wedge y \in X_{n} \setminus X_{n-1}) \text{ or } (y \in X_{n-1} \setminus \mrm{ran}(f_s) \wedge x \in X_{n} \setminus X_{n-1});$$
\end{enumerate}	
	\item\label{item_meet_tree} $(\mrm{ran}(f^i_t) \setminus X_{n-1} : t \in T_n \setminus T_{< n}, i \in \{1, -1 \})$ is a sequence of \mbox{pairwise disjoint sets;}
	\item\label{itemh_co} $X_{n+1} \supsetneq X_n \cup \{\mrm{ran}(f^i_t) : t \in T_{n+1} \setminus T_n, i \in \{1, -1 \}\}$;
	\item\label{graph} we define the graph $(\mrm{seq}_k(X), R_k)$ as follows:
	$$\bar{x}  R_k \bar{y} \text{ iff } \exists t \in T, \exists i \in \{1, -1\}  \text{ such that } f^i_t(\bar{x}) = \bar{y};$$
	notice that $\bar{x}  R_k \bar{y}$ implies $\bar{y}  R_k \bar{x}$ and that $\neg (\bar{x} R_k \bar{x})$ is by (\ref{item_no_fix_point})$(\beta)$.
	\item given $\bar{x}, \bar{y} \in \mrm{seq}_k(X)$ we let $\bar{x} E_k \bar{y}$ when $\bar{x}, \bar{y}$ are in the same connected component in the graph $(\mrm{seq}_k(X), R_k)$;
	\item\label{A2_no_cycles} the graph $(\mrm{seq}_k(X), R_k)$ has no cycles.
\end{enumerate}
\end{definition}

\begin{convention}\label{the_m_convention} $\mathfrak{m} = (X_n, \bar{f}_n : n < \omega) \in \mathrm{K}^{\mrm{ri}}(T)$ (cf. Definition~\ref{hyp_A2}).
\end{convention}

	\begin{notation}\label{notation_<_m} Let $\mathfrak{m} \in \mathrm{K}^{\mrm{ri}}(T)$.
	\begin{enumerate}[(1)]
	\item\label{notation_n(x)} For $x \in X$, we let $\mathbf{n}(x)$ be the unique $n < \omega$ such that $x \in X_n \setminus X_{< n}$.
	\item For $k \geq 1$ and $\bar{x} = (x_0, ..., x_{k-1}) \in \mrm{seq}_k(X)$, let $\mathbf{n}(\bar{x}) = \mrm{max}\{\mathbf{n}(x_\ell) : \ell < k\}$.
	\item For $\bar{x} \in \mrm{seq}_k(X)$ we let:
	$$\mrm{suc}(\bar{x}) = \{\bar{y} \in \mrm{seq}_k(X): \mathbf{n}(\bar{x}) < \mathbf{n}(\bar{y}) \text{ and } \exists t \in T, \, \exists i \in \{1, -1\} \text{ s.t. } f^i_t(\bar{x}) = \bar{y}\}.$$
	\item\label{<_mathfrak} $<^k_{\mathfrak{m}}$ is the transitive closure of the relation $\bar{y} \in \mrm{suc}(\bar{x})$ on $\mrm{seq}_k(X)$.
	\item\label{def_reasonable_item} We say that $\bar{x} \in \mrm{seq}_k(X)$ is reasonable when the following happens:
	$$n_1 < n_2, \, x_{i(1)} \in X_{n(1)} \setminus  X_{< n(1)}, \, x_{i(2)} \in X_{n(2)} \setminus  X_{< n(2)} \Rightarrow i(1) \leq i(2).$$
	\item\label{k=1} If $k = 1$, we may simply write $<_{\mathfrak{m}}$ instead of $<^1_{\mathfrak{m}}$. Also, we may simply write $x <_{\mathfrak{m}} y$ instead of $(x) <_{\mathfrak{m}} (y)$.
	\item\label{f_eta_pre} Assume $n < \omega$, $\bar{t} \in T^n$ and $\eta \in \{1, -1\}^n$, then we let:
	$$f_{(\bar{t}, \eta)} = (\cdots \circ f^{\eta(\ell)}_{t_\ell} \circ \cdots)_{\ell < n}.$$
	\item The pair $(\bar{t}, \eta)$, where $\eta \in \{1, -1\}^n$ and $\bar{t} \in T^n$:
	\begin{enumerate}[(a)]
	\item is called {\em reduced} when $\ell < n-1$ and $t_\ell = t_{\ell+1}$ implies that $\eta(\ell) = \eta(\ell+1)$;
	\item is called {\em reduced} for $x \in X$ when the sequence $(f_{(\bar{t} \restriction \ell, \eta \restriction \ell)}(x) : \ell \leq n)$ is well-defined and with no repetitions (so $(\bar{t}, \eta)$ is reduced, but the converse may fail, e.g. we might have that $t_1 < t_2$, $x \in \mrm{dom}(f_{t_1})$ and $f^{-1}_{t_2} \circ f_{t_1}(x) = x$);
	\item is called {\em strongly reduced} for $x$ when it is reduced for $x$ and letting, for  $\ell \leq n$, $x_\ell = f_{(\bar{t} \restriction \ell, \eta \restriction \ell)}(x)$ we have that for all $s <_T t_\ell$, $x_\ell \notin \mrm{dom}(f^{\eta(\ell)}_s)$.
	\end{enumerate}
	\end{enumerate}
\end{notation}

	\begin{observation}\label{observation_new_elements_are_moved} 
	In the context of Definition~\ref{hyp_A2}, we have: let $i \in \{1, -1\}$ and $t \in T$, then for every $x \in \mrm{dom}(f_t)$ we have that $x \neq f^{i}_t(x)$, further there is a unique $0 < n < \omega$ such that $x \in X_{n-1} \setminus X_{< n-1}$ and $f^i_t(x) \in X_n \setminus X_{n-1}$ or $f^i_t(x) \in X_{<n-1}$.
\end{observation}

\begin{proof}  Easy.
\end{proof}


	\begin{cclaim}\label{K_co_non-empty} For $T$ as in Hypothesis~\ref{hyp_A1}, $\mathrm{K}^{\mrm{ri}}(T) \neq \emptyset$ (cf. Definition~\ref{hyp_A2}).
\end{cclaim}

	\begin{proof} Straightforward.
\end{proof}

	\begin{cclaim}\label{claim_for_star} Let $k \geq 1$ and $\bar{x}, \bar{y} \in \mrm{seq}_k(X)$.
	\begin{enumerate}[(A)]
	\item If $f^i_t(\bar{x}) = \bar{y}$, then:
	\begin{enumerate}[$(\cdot_1)$]
	\item $\mathbf{n}(\bar{x}) \in \{0, ..., \mathbf{n}(\bar{y})-1\} \cup \{ \mathbf{n}(\bar{y})+1\}$;
	\item $\mathbf{n}(\bar{x}) = \mathbf{n}(\bar{y}) - 1 \; \Leftrightarrow \; \bar{x} <^k_{\mathfrak{m}} \bar{y}$.
	\end{enumerate}
	\item If $f^i_t(\bar{x}) = \bar{y}$ and $\mathbf{n}(\bar{x}) < \mathbf{n}(\bar{y})$, then:
	$$\bar{x}' <^k_{\mathfrak{m}} \bar{y} \wedge \mathbf{n}(\bar{x}') = \mathbf{n}(\bar{x}) \; \Rightarrow \; \bar{x}' = \bar{x}.$$
	\item\label{claim_for_starC} If $\bar{x} E_k \bar{y}$, then there are $n \geq 0$, $\bar{x}_0, ..., \bar{x}_n$, $\bar{t} = (t_\ell : \ell < n)$, $\eta \in \{1, -1\}^n$ s.t.:
	\begin{enumerate}[(a)]
	\item $\bar{x} = \bar{x}_0$ and $\bar{y} = \bar{x}_n$;
	\item if $\ell < n$, then $f^{\eta(\ell)}_{t_\ell}(x_\ell) = x_{\ell+1}$; 
	\item if $\bar{x} \leq^k_{\mathfrak{m}} \bar{y}$ and $\ell < n$, then $\mathbf{n}(x_\ell) < \mathbf{n}(x_{\ell+1})$ (cf. Notation~\ref{notation_<_m}(\ref{notation_n(x)}));
	\item there is no $s < t_\ell$ such that $x_\ell \in \mrm{dom}(f^{\eta(\ell)}_s)$.
\end{enumerate}	 
	\item\label{claim_for_starD} In clause (C), if $\bar{x} \leq^k_{\mathfrak{m}} \bar{y}$, then $\bar{t}, \eta$ and $(\bar{x}_\ell : \ell \leq n)$ are unique, $\mathbf{n}(t_\ell) < \mathbf{n}(t_{\ell+1})$ and $x_{\ell+1} \in \mrm{suc}(x_\ell)$, for $\ell < n$.
	\item\label{claim_for_starE} If $x \in X$, $\bar{t} \in T^m$, $\eta \in \{1, -1\}^m$ and $x_0 = x$ and $x_\ell = f_{(\bar{s} \restriction \ell, \eta \restriction \ell)}(x)$, for $0 < \ell \leq m$, then for some $k \geq 1$, $\bar{i} = (i_0, ..., i_k)$ and $\bar{s}$ we have:
	\begin{enumerate}[$(\cdot_1)$]
	\item $0 = i_0 < \cdots < i_k \leq m$;
	\item $x_{i_\ell} = x_{i_{\ell+1} - 1}$, for $\ell < k$;
	\item $x_{i_k} = x_m$;
	\item for some $k_1 \leq k$ we have:
	\begin{enumerate}[(i)]
	\item for all $\ell < [k_1, k)$, $x_{i_{\ell+1}} \in \mrm{suc}(x_{i_\ell})$;
	\item for all $\ell < k_1$, $x_{i_{\ell}} \in \mrm{suc}(x_{i_\ell+1})$;
	\item $\bar{s} = (s_\ell : \ell \leq k)$, $s_\ell \leq_T t_\ell$, $f^{\eta_{i_\ell}}_{s_\ell}(x_{i_\ell}) = x_{i_{\ell+1}}$, for $\ell < k$ and ${(\bar{s}, (\eta(i_\ell)  : \ell < k))}$ is strongly reduced for $x = x_0$.
\end{enumerate}	
	\end{enumerate}
	\item Notice that in (E) we have that:
	\begin{enumerate}[($\cdot_1$)]
	\item the interval $(x_{i_\ell}, ..., x_{i_{\ell+1} - 1})$ is canceled to $x_{i_\ell}$;
	\item $\bar{i}$ and $\bar{s}$ are determined uniquely by $(x, \bar{t}, \nu)$. 
	\end{enumerate}
	\item\label{claim_for_starG} 
	\begin{enumerate}[$(\cdot_1)$]
	\item If $(\bar{t}, \eta) \in T^n \times \{1, -1\}^n$, $x \in \mrm{dom}(f_{(\bar{t}, \eta)})$ and $x/E_1$ is disjoint from $X_m$, then $X_m \subseteq \mrm{dom}(f_{(\bar{t}, \eta)})$;
	\item if $x, y \in X$ are $E_1$-equivalent, then there is a unique pair $(\bar{t}, \eta)$ such that $f_{(\bar{t}, \eta)}(x) = y$ and $(\bar{t}, \eta)$ is strongly reduced for $x$;
	\item if $f_{(\bar{t}, \eta)}(x) = y$ and $(\bar{t}, \eta)$ is strongly reduced for $x$, then $(\bar{t}, \eta)$ is unique;
	\item if $\bar{t}_1 = (t_{(1, \ell)} : \ell < n)$, $(\bar{t}_1, \eta_1)$ is strongly reduced for $x$ and $x/E_1 \cap X_m = \emptyset$, then, for every $\ell < n$, $t_{(1, \ell)} \notin T_m$.
	\end{enumerate}
	\item\label{claim_for_starH} If $\iota = 1, 2$, $(\bar{t}_\iota, \eta_\iota) \in T^{n(\iota)} \times \{1, -1\}^{n(\iota)}$, $(\bar{t}_1, \eta_1)$ is strongly reduced for $x$, $x \in \mrm{dom}(f_{(\bar{t}_1, \eta_1)}) \cap \mrm{dom}(f_{(\bar{t}_2, \eta_2)})$ and $x/E_1 \cap X_m = \emptyset$, then:
	\begin{enumerate}[$(\cdot_1)$]
	\item $f_{(\bar{t}_1, \eta_1)} \restriction X_m = f_{(\bar{t}_2, \eta_2)} \restriction X_m$;
	\item for every $\ell < n(1)$ there is $j < n(2)$ such that we have the following:
	$$t_{(1, \ell)} \leq_T t_{(2, j)}, \; \eta_1(\ell) = \eta_2(j), \; t_{(1, \ell)} \notin T_m,$$
where we let $\bar{t}_1 = (t_{(1, \ell)} : \ell < n(1))$ and $\bar{t}_2 = (t_{(2, j)} : j < n(2))$.
	\end{enumerate}
	\end{enumerate}
\end{cclaim}

	\begin{proof} Clauses (A), (B), (C) are immediate by the definitions. Clauses (E) and (F) can be proved by induction on $\mrm{lg}(\bar{t})$. Clause (G) can be proved using (E) and (F). Finally, we give some details on clause (H). To this extent, we apply clause (E) with $(x, \bar{t}_2, \eta_2)$ here standing for $(x, \bar{t}, \nu)$ there and thus get $\bar{s}, k, k_1, \bar{i}$. So, in particular, letting $\eta'_2 = (\eta_{(2, i_\ell)} : \ell \leq k)$ we have the following:
\begin{enumerate}[$(\cdot_1)$]
	\item $(\bar{s}, \eta'_2)$ is strongly reduced for $x$;
	\item $f_{(\bar{s}, \eta'_2)}(x) = f_{(\bar{t_2}, \eta_2)}(x) = f_{(\bar{t_1}, \eta_1)}(x)$.
\end{enumerate}
Recall now that:
\begin{enumerate}[$(\cdot_3)$]
	\item $(\bar{t}_1, \eta_1)$ is strongly reduced for $x$.
\end{enumerate}
Thus, by $(\cdot_1)$-$(\cdot_3)$ we can apply clause (G)$(\cdot_2)$, thus getting:
\begin{enumerate}[$(\cdot_4)$]
	\item $(\bar{s}, \eta'_2) = (\bar{t}_1, \eta_1)$, and so $\mrm{lg}(\bar{t}_2) = \mrm{lg}(\bar{s}) = k$.
\end{enumerate}
So, if $\ell \leq k$, then $s_\ell = t_{(1, \ell)} \leq_T t_{(2, i_\ell)}$, and so we are done. Finally, the fact that for $\ell < n(1)$ we have that $t_{(1, \ell)} \notin T_m$ is by clause $(G)(\cdot_4)$ of the present claim.
\end{proof}

	\begin{observation}\label{observation_Xtree}
	\begin{enumerate}[(1)]
	\item (Recalling \ref{notation_<_m}(\ref{k=1})) $(X, <_\mathfrak{m})$ is a tree with $\omega$ levels.
	\item Every $z \in X_0$ is a root of the tree $(X, <_\mathfrak{m})$.
	\item If $y \in X_{n+1} \setminus X_n$, then for at most one $x \in X_n$ we have $y \in \mrm{suc}(x)$. 
	\item\label{wlog_reasonable} If $\bar{x} \in \mrm{seq}_k(X)$, then some permutation of $\bar{x}$ is reasonable.
	\item For every $k \geq 1$, $(\mrm{seq}_k(X), <^k_{\mathfrak{m}})$ is a tree, with possibly more than one root.
	\item For every $\bar{x}, \bar{y} \in \mrm{seq}_k(X)$, if $\bar{x}, \bar{y}$ admit a common lower bound we denote by $\bar{x} \wedge \bar{y}$ their maximal common lower bound (recall $(\mrm{seq}_k(X), <^k_{\mathfrak{m}})$ is a tree).
	\item\label{reasonable_implies_reasonable} If $\bar{x}$ is reasonable and $\bar{x} \leq^k_\mathfrak{m} \bar{y}$, then $\bar{y}$ is reasonable.
	\item\label{observation_Xtree_item8} If $\bar{x} \in \mrm{seq}_k(X)$ is reasonable, $\bar{x} \leq^k_\mathfrak{m} \bar{y}^1 = (y^1_0, ..., y^1_{k-1})$, $\bar{x} \leq^k_\mathfrak{m} \bar{y}^2 = (y^2_0, ..., y^2_{k-1})$ and $y^1_{k-1} = y^2_{k-1}$, then $ \bar{y}^1 = \bar{y}^2$.
	\item If $\bar{x} \in \mrm{seq}_k(X)$, $\bar{y}$ is a root of $(\bar{x}/E_k, <^k_\mathfrak{m})$ and  $\bar{y}$ is reasonable, then so is $\bar{x}$.
	\end{enumerate}
\end{observation}

	\begin{proof} This is simply traking down the definition. E.g. item (3) is by \ref{hyp_A2}(\ref{item_no_fix_point})-(\ref{item_meet_tree}). 
\end{proof}


	\begin{definition}\label{def_G02_cohopfian} Let $\mathfrak{m} \in \mathrm{K}^{\mrm{ri}}(T)$ (i.e. as in Convention~\ref{the_m_convention}).
	\begin{enumerate}[(1)]
	\item Let $G_2 = G_2[\mathfrak{m}]$ be $\bigoplus \{ \mathbb{Q}x : x \in X\}$.
	\item Let $G_0 = G_0[\mathfrak{m}]$ be the subgroup of $G_2$ generated by $X$, i.e. $\bigoplus \{ \mathbb{Z}x : x \in X\}$.
	\item\label{hatf_on_G2} For $t \in T$ and $i \in \{1, -1\}$, let:
	\begin{enumerate}[(a)]
	\item $H^i_{(2, t)} = \bigoplus \{ \mathbb{Q}x : x \in \mrm{dom}(f^i_t)\}$;
	\item $I^i_{(2, t)} = \bigoplus \{ \mathbb{Q}x : x \in \mrm{ran}(f^i_t)\}$;
	\item\label{hat_f} $\hat{f}^{(i, 2)}_t$ is the (unique) isomorphism from $H^i_{(2, t)}$ onto $I^i_{(2, t)}$ such that $x \in \mrm{dom}(f^i_t)$ implies that $\hat{f}^{(i, 2)}_t(x) = f^i_t(x)$ (cf. Definition~\ref{hyp_A2}(\ref{def_of_f_t})).
	\end{enumerate}
	\item For $i \in \{1, -1\}$ and $t \in T$, we define $H^i_{(0, t)} := H^i_{(2, t)} \cap G_0$ and $I^i_{(0, t)} := I^i_{(2, t)} \cap G_0$.
	\item For $\hat{f}^{(i, 2)}_t$ as above, we have that $\hat{f}^{(i, 2)}_t [H^i_{(0, t)}] = I^i_{(0, t)}$. We define:
	$$\hat{f}^{(i, 0)}_t = \hat{f}^{(i, 2)}_t \restriction H^i_{(0, t)}.$$
	\item\label{n[a]} For $a \in G^+_2$, let $\mathbf{n}(a)$ be the minimal $n < \omega$ such that $a \in \langle X_{n}\rangle^*_{G_2}$.
	\item\label{def_<*} We define the partial order $<_*$ on $G^+_0$ by letting $a <_* b$ if and only if $a \neq b \in G^+_0$ and, for some $0 < n < \omega$,  $a_0, ..., a_n \in G_0, a_0 = a, a_n = b$ and:
	$$\ell < n \Rightarrow \exists i \in \{1, -1\} \, \exists t \in T(\hat{f}^{(i, 0)}_t(a_\ell) = a_{\ell + 1}) \text{ and }$$ 
	$$\mathbf{n}(a_\ell) < \mathbf{n}(a_{\ell+1}).$$
	\item For $a = \sum_{\ell < m}q_\ell x_\ell$, with $(x_\ell : \ell < m) \in \mrm{seq}_m(X)$ and $q_\ell \in \mathbb{Q}^+$, we let:
	$$\mathrm{supp}(a) = \{ x_\ell : \ell < m \}.$$
	\item\label{the_mathcal_E} $\mathcal{E}_{\mathfrak{m}}$ is the equivalence relation on $G_2^+$ which is the closure to an equivalence relation of the partial order $\leq_*$.
	\item\label{f_eta} Assume $n < \omega$, $\bar{t} \in T^n$, $\eta \in \{1, -1\}^n$, $\iota \in \{0, 2\}$, then we let:
	$$\hat{f}_{(\bar{t}, \eta)}^\iota = (\cdots \circ \hat{f}^{(\eta(\ell), \iota)}_{t_\ell} \circ \cdots)_{\ell < n}.$$
\end{enumerate}
\end{definition}

%

	\begin{lemma}\label{G0_is_tree}\begin{enumerate}[(1)]
	\item\label{G0_is_tree_restrictiontoX}  $<_* \restriction X = <^1_{\mathfrak{m}} = <_{\mathfrak{m}} $ (cf. \ref{notation_<_m}(\ref{<_mathfrak})(\ref{k=1})).
	\item $(G_0, <_*)$ is a tree with $\omega$ levels (recall Hypothesis~\ref{hyp_A1}(\ref{omega_levels})).
	\item\label{G0_is_tre_item4} If $s \leq_T t$, then $\hat{f}^{(i, \ell)}_s \subseteq \hat{f}^{(i, \ell)}_t$, for $\ell \in \{0, 2 \}$, $i \in \{1, -1\}$.
	\item\label{n[a]1} If $t \in T$, $\hat{f}^{(i, 2)}_t(a) = b$, $a \in G^+_0$, then $\mathbf{n}(a) \neq \mathbf{n}(b)$.
	\item If $a <_* b$ (so $a, b \in G^+_0$), then the sequence $(a_\ell : \ell \leq n)$ from \ref{def_G02_cohopfian}(\ref{def_<*}) is unique.
	\item If $a_1 <_* a_2$, and, for $\ell \in \{ 1, 2 \}$, $a_\ell = \sum_{i < k} q_ix^\ell_i$, $q_i \in \mathbb{Q}^+$, $\bar{x}^\ell = (x^\ell_i : i < k) \in \mrm{seq}_k(X)$, then maybe after replacing $\bar{x}^1$ with \mbox{a permutation of it, $\bar{x}^1 \leq^k_\mathfrak{m} \bar{x}^2$.}
\end{enumerate}
\end{lemma}

	\begin{proof} Unraveling definitions, e.g. for item (3) use Definition~\ref{hyp_A2}(\ref{item_fs_sub_ft}), we elaborate only on item (4). To this extent, let $a, b$ be as there. W.l.o.g. $a <_* b$, as can replace $b$ with $a$, depending on the value of $i$ in $\hat{f}^{(i, 2)}_t$.
As $a \neq 0$ (recall $a \in G^+_0$), let $a = \sum_{j \leq n} q_j x_j$, with $(x_j : j \leq n)$ with no repetitions and $q_j \in \mathbb{Q}^+$. 
	W.l.o.g. $\mathbf{n}(x_j) \leq \mathbf{n}(x_{j+1})$, for $j < n$ (cf. Observation~\ref{observation_Xtree}(\ref{wlog_reasonable})). 
	Clearly $a \in \langle X_{\mathbf{n}(x_n)} \rangle^*_{G_2}$ but $a \notin \langle X_{<\mathbf{n}(x_n)} \rangle^*_{G_2}$. As $\hat{f}^{(i, 2)}_t(a)$ is well-defined, clearly $\{x_j : j \leq n \} \subseteq \mrm{dom}(f^{(i, 2)}_t)$ and $b = \hat{f}^{(i, 2)}_t(a) = \sum_{j \leq n} q_j f^{(i, 2)}_t(x_j)$ and, as $f^{(i, 2)}_t$ is $1$-to-$1$, the sequence $(f^{(i, \ell)}_t(x_j) : j \leq n)$ is with no repetitions. 
	By Observation~\ref{observation_new_elements_are_moved} applied with $n$ there as $\mathbf{n}(x_n)$ here, $f_t(x_n) \notin \langle X_{\mathbf{n}(x_n)} \rangle^*_{G_2} $, hence we have that  $\mathbf{n}(b) \neq \mathbf{n}(a)$, so (4) holds.
\end{proof}

	\begin{cclaim}\label{claim9A} If (A), then (B), where:
	\begin{enumerate}[(A)]
	\item
	\begin{enumerate}[(a)]
	\item $a, b_\ell \in G^+_0$, for $\ell < \ell_*$ and $\ell_* \geq 2$;
	\item $a \mathcal{E}_{\mathfrak{m}} b_\ell$ and the $b_\ell$'s are with no repetitions;
	\item $a = \sum \{ q_i x_i : i < j \}$ and $j > 1$;
	\item $\bar{x} = (x_i : i  < j) \in X^j$ is injective and reasonable;
	\item $q_i \in \mathbb{Z}^+$;
	\item $\bar{x}$ is $<^j_{\mathfrak{m}}$-minimal;
	\end{enumerate}
	\item for every $\ell < \ell_*$ there are $\bar{y}^\ell = (y_{(\ell, i)} : i < j)$ such that:
	\begin{enumerate}
	\item\label{Ba} $y_{(\ell, i)} = : y^\ell_i \in X$ and $\bar{x} \leq^j_{\mathfrak{m}} \bar{y}^\ell$;
	\item $b_\ell = \sum \{ q_i y_{(\ell, i)} : i < j \}$, and so the $\bar{y}^\ell$'s are pairwise distinct;
	\item $(y_{(\ell, i)} : i < j)$ is injective and reasonable;
	\item there are $\ell_1 \neq \ell_1 < \ell_*$ and $i_1, i_2 < j$ such that:
	\begin{enumerate}[(i)]
	\item if $\ell < \ell_*$, $i < j$ and $y_{(\ell, i)} = y_{(\ell_1, i_1)}$, then $(\ell, i) = (\ell_1, i_1)$;
	\item if $\ell < \ell_*$, $i < j$ and $y_{(\ell, i)} = y_{(\ell_2, i_2)}$, then $(\ell, i) = (\ell_2, i_2)$.
	\end{enumerate}
	\item\label{claim9A_itemBf} $(y_{(\ell, j-1)} : \ell < \ell_*)$ is without repetitions.
	\end{enumerate}
	\end{enumerate}
\end{cclaim}

\begin{proof} By the def. of $\leq_*$ there are $(y_{(\ell, i)} : i < j, \, \ell < \ell_*)$ satisfying clauses $(a)$-$(c)$ of $(B)$, e.g., $\bar{y}^\ell$ is reasonable by \ref{observation_Xtree}(\ref{reasonable_implies_reasonable}) and clause (A)(f). Also, clause (e) holds by \ref{observation_Xtree}(\ref{observation_Xtree_item8}), recalling that the $\bar{y}^\ell$'s are pairwise distinct. Recall now that $(\{\bar{y} \in \mrm{seq}_j(X): \bar{x} \leq^j_X \bar{y}\}, \leq^j_\mathfrak{m})$ is a tree. There are two cases. 
\newline \underline{Case 1}. $\{\bar{y}^\ell : \ell < \ell_* \}$ is not linearly ordered by $\leq^j_\mathfrak{m}$.
\newline  Then there are $\ell(1) \neq \ell(2) < \ell_*$ such that $\bar{y}^{\ell(1)}, \bar{y}^{\ell(2)}$ are locally $\leq^j_\mathfrak{m}$-maximal (i.e., with respect to $\{\bar{y}^\ell : \ell < \ell_* \}$). 
Using the assumption that the sequences are reasonable we can choose $i_1 = j-1 = i_2$, and then by \ref{observation_Xtree}(\ref{observation_Xtree_item8}) we are done.
\newline \underline{Case 2}. Not Case 1.
\newline So w.l.o.g. we have that, for every $\ell < \ell_*-1$, $\bar{y}^\ell <^j_\mathfrak{m}\bar{y}^{\ell+1}$. Let: 
\begin{enumerate}[$(\cdot_1)$]
	\item $i(1) < j$ be such that $i < j$ implies $\mathbf{n}(y^0_i) \geq \mathbf{n}(y^0_{i(1)})$;
	\item $i(2) < j$ be such that $i < j$ implies $\mathbf{n}(y^{\ell_* - 1}_{i(2)}) \geq \mathbf{n}(y^{\ell_* - 1}_{i})$.
\end{enumerate}
Then $(0, i(1))$, $(\ell_*-1, i(2))$ are as required. As, for $\ell < \ell_*$, $\bar{y}^\ell$ is assumed to be reasonable we can actually choose $i(1), i(2)$ such that  $i(1) = 0$ and $i(2) = j-1$. 
\end{proof}

	\begin{definition}\label{def_G1_cohopfian} Let $(p_a : a \in G^+_0)$ be a sequence of pairwise distinct primes s.t.:
$$a = \sum_{\ell < k} q_\ell x_\ell, \, q_\ell \in \mathbb{Z}^+, \, (x_\ell : \ell < k) \in \mrm{seq}_k(X) \Rightarrow p_a \not \vert \; q_\ell.$$ 
	\begin{enumerate}[(1)] 
	\item\label{P_a} For $a \in G^+_0$, let (recalling the definition of $\mathcal{E}_{\mathfrak{m}}$ from Definition~\ref{def_G02_cohopfian}(\ref{the_mathcal_E})):
	$$\mathbb{P}_a = \{ p_b : b \in G^+_0, b \in a/\mathcal{E}_{\mathfrak{m}} \}.$$ 
	\item\label{def_G1} Let $G_1 = G_1[\mathfrak{m}] = G_1[\mathfrak{m}(T)] = G_1[T]$ be the subgroup of $G_2$ generated by:
$$\{ m^{-1} a : a \in G^+_0, \; m \in \omega \setminus \{ 0 \} \text{ a product of primes from $\mathbb{P}_a$, poss. with repetitions}\}.$$
	\item\label{def_G1p} For a prime $p$, let $G_{(1, p)} = \{ a \in G_1 : a \text{ is divisible by $p^m$, for every $0 < m < \omega$}\}$ (notice that, by Observation~\ref{generalG1p_remark}, $G_{(1, p)}$ is always a pure subgroup of $G_1$).
\end{enumerate}
\end{definition}

	\begin{lemma}\label{lemma_pre_main_th}\begin{enumerate}[(1)]
	\item\label{G1p_pure_co-hop} If $p = p_a$, $a \in G^+_1$, then:
	$$G_{(1, p)} = \langle b \in G^+_0 : a \mathcal{E}_{\mathfrak{m}} b \rangle^*_{G_1}.$$
	\item For $t \in T$ and $i \in \{1, -1\}$, we let $H^{i}_{(1, t)} := H^i_{(2, t)} \cap G_1$ and $I^i_{(1, t)} := I^i_{(2, t)} \cap G_1$.
	\item\label{convention_hatf_G1} For $\hat{f}^{(i, 2)}_t$ as in Def.~\ref{def_G02_cohopfian}(\ref{hat_f}), $\hat{f}^{(i, 2)}_t [H^i_{(1, t)}] = I^i_{(1, t)}$. We let $\hat{f}^{(i, 1)}_t = \hat{f}^{(i, 2)}_t \restriction H^i_{(1, t)}$.
	\item For $\ell \in \{0, 1, 2\}$, if $s <_T t$, then $H^i_{(\ell, s)} \leq_* H^i_{(\ell, t)}$ (recall \ref{def_pure});
	\item\label{onto} For $\ell \in \{0, 1, 2\}$, if, for every $n < \omega$, $t_n <_T t_{n+1}$, then we have:
$$\bigcup_{n < \omega} H^i_{(\ell, t_n)} = \bigcup_{n < \omega} I^i_{(\ell, t_n)} = G_\ell.$$
	\item\label{onto+} $f^{(i, 1)}_t$ maps $H^i_{(1, t)}$ onto $I^i_{(1, t)}$.
\end{enumerate}
\end{lemma}

	\begin{proof} (1) is proved similarly to \cite[5.17(1)]{1205}. The rest is easy, recalling that in \ref{hyp_A2}(\ref{def_of_f_t}) we ask that $X_{n-1} = \mrm{dom}(f_t) \cap \mrm{ran}(f_t)$.
\end{proof}

	\begin{notation} When we omit $i$ in objects that depend on $i$, such as $H^i_{(2, t)}$ etc., we mean that $i = 1$. So e.g. by $\hat{f}^2_{t}$ we mean $\hat{f}^{(1, 2)}_{t}$. This should be kept in mind.
\end{notation}

	\begin{theorem}\label{main_th_Sec1} Let $\mathfrak{m}(T) \in \mathrm{K}^{\mrm{ri}}(T)$.
	\begin{enumerate}[(1)]
	\item We can modify the construction so that $G_1[\mathfrak{m}(T)] = G_1[T]$ has domain $\omega$ and the function $T \mapsto G_1[T]$ is Borel (for $T$ a rooted tree with domain $\omega$).
	\item If $T$ is not well-founded then $G_1[T]$ has a non-trivial free autom. (cf. \ref{def_free_auto}).
	\item If $T$ is well-founded then $G_1[T]$ has only trivial onto endomorphisms.
\end{enumerate}
\end{theorem}

	\begin{proof} Item (1) is easy. Concerning item (2), let $(t_n : n < \omega)$ be an infinite branch of $T$ and $t_n <_T t_{n+1}$, so $t_{n+1} \notin T_n$. By Lemma~\ref{G0_is_tree}(\ref{G0_is_tre_item4}), $(\hat{f}^2_{t_n} : n < \omega)$ is increasing, by Definition~\ref{def_G02_cohopfian}(\ref{hat_f}), $\hat{f}^2_{t_n}$ is an isomorphism of $H_{(2, t_n)}$ onto $I_{(2, t_n)}$, thus $\hat{f}^2 = \bigcup_{n < \omega} \hat{f}^2_{t_n}$ is an isomorphism of $G_2$ onto $G_2$, since $(H_{(2, t_n)} : n < \omega)$ and $(I_{(2, t_n)} : n < \omega)$ are chain of pure subgroups of $G_2$ with limit $G_2$, because, by Definition~\ref{hyp_A2}(\ref{def_of_f_t}), we have:
	$$H_{(2, t_n)} \supseteq \mrm{dom}(f_{t_n}) \subseteq \mrm{dom}(f_{t_{n+1}}) \subseteq H_{(2, t_{n+1})},$$
	$$I_{(2, t_n)} \supseteq \mrm{ran}(f_{t_n}) \subseteq \mrm{ran}(f_{t_{n+1}}) \subseteq I_{(2, t_{n+1})}$$
and by \ref{lemma_pre_main_th}(\ref{onto}) we have that $\bigcup_{n < \omega} H_{(2, t_n)} = G_2 = \bigcup_{n < \omega} I_{(2, t_n)}$.
Thus, by \ref{lemma_pre_main_th}(\ref{onto+}):
$$\hat{f}^1 := \hat{f}^2 \restriction G_1 = \bigcup_{n < \omega} \hat{f}^1_{t_n} = \bigcup_{n < \omega} \hat{f}^2_{t_n} \restriction H_{(1, t_n)}$$ 
$$\hat{f}^{-1} = \bigcup_{n < \omega} \hat{f}^{(-1, 1)}_{t_n} = \bigcup_{n < \omega} \hat{f}^{(-1, 2)}_{t_n} \restriction H_{(1, t_n)}$$ 
and so $\hat{f}^1$ is an isomorphism of $G_1$ onto $G_1$. Clearly $\hat{f}^1$ is a non-trivial free automorphism of $G_1$ (cf. \ref{def_free_auto}), and so item (2) of the theorem is proved.

\medskip
\noindent We now prove item (3). Suppose that $(T, <_T)$ is well-founded and suppose that $\pi \in \mrm{End}(G_1)$ is onto, letting $G_1 = G_1[T]$. We shall show that $\pi \in \{\mrm{id_{G_1}}, - \mrm{id_{G_1}}\}$.

\noindent	Assume $k \geq 1$, $\bar{x} \in \mrm{seq}_k(X)$ and $q_\ell \in \mathbb{Q}^+$, for $\ell < k$ and $a = \sum_{\ell < k} q_\ell x_\ell$. Then:
	\begin{enumerate}[$(*_1)$]
	\item 
	\begin{enumerate}[(a)]
	\item There is $j < \omega$ and for every $i < j$ there is $\bar{x}^i \in \bar{x}/E_k$ and $q^i \in \mathbb{Q}^+$ such that the $\bar{x}^i$'s are pairwise distinct and we have:
	 $$\pi(\sum_{\ell < k} q_\ell x_\ell) = \sum_{i < j} q^i \sum_{\ell < k} q_\ell x^i_\ell;$$
	\item let $\bar{x}' \in \bar{x}/E_k$ be $<^k_{\mathfrak{m}}$-minimal, then, for every $i < j$, $\bar{x}' \leq^k_{\mathfrak{m}} \bar{x}^i$;
	\item for every $i < j$ there are $n_{i}, \bar{t}_{i}, \eta_{i}$ such that  $\bar{t}_{i} \in T^{n_{i}}, \eta_{i} \in \{1, -1\}^{n_{i}}$ and:
	$$\pi(\sum_{\ell < k} q_\ell x_\ell) = \sum \{ q^i q_\ell f_{(\bar{t}_{i}, \eta_{i})}(x'_\ell) : \ell < k, i < j\},$$
	where, as in \ref{notation_<_m}(\ref{f_eta_pre}), $f_{(\bar{t}_{i}, \eta_{i})} = (\cdots \circ f^{\eta_{i}(r)}_{t_{(i, r)}} \circ \cdots)_{r < n(i)}$, and if $r + 1 < n(i)$ and $\eta_{i}(r) = - \eta_{i}(r+1)$, then $t_{(i, r)} \neq t_{(i, r+1)}$ and $\mathbf{n}(t_{(i, r)})$ is minimal.
	\end{enumerate}
\end{enumerate}
[Why? Clause (a) is by the choice of $G_1$ and $\mathbb{P}_a$, for $a \in G^+_0$. \mbox{Clauses (b), (c) follow.]}
\begin{enumerate}[$(*_2)$]
	\item If $a \in G_1$, $x \in X$ and $\pi(a) \in \mathbb{Q}^+x$, then $\mrm{supp}(a)$ is a singleton.
\end{enumerate}
Why $(*_2)$? Let $a \in G^+_1$, $k \geq 1$, $\bar{x} \in \mrm{seq}_k(X)$ and $q_\ell \in \mathbb{Q}^+$, for $\ell < k$ and $a = \sum_{\ell < k} q_\ell x_\ell$. Let $p = p_a$, so $a \in G_{(1, p)}$ and thus $\pi(a) \in G_{(1, p)}$. Note that $a \neq 0$. Now apply $(*_1)$ to $a$, then there are $j < \omega$ and, for $i < j$, $\bar{x}^i$, $q^i$ such that:
	$$\pi(a) = \sum_{i < j} q^i \sum_{\ell < k} q_\ell x^i_\ell.$$
So, by our assumptions, for some $q \neq 0$, $\pi(a) = qx$, thus:
 	$$qx = \sum_{i < j} q^i \sum_{\ell < k} q_\ell x^i_\ell.$$
Now, the LHS of the equation has support a singleton, but, if $k \geq 2$, by \ref{claim9A}, the RHS of the equation has support with at least two members, so $k \leq 1$, and, as $a \neq 0$, necessarily, $k = 1$, as promised.

\begin{enumerate}[$(*_3)$]
	\item There are no $x \in X$ and $y_1 \neq y_2 \in X$ such that $\bigwedge_{\ell = 1, 2} \pi(y_\ell) \in \mathbb{Q}^+x$.
\end{enumerate}
[Why? As then for some $m \in \{1, 2\}$, $\pi(y_1 + my_2) \in \mathbb{Q}^+x$, contrary to $(*_2)$.]
\begin{enumerate}[$(*_4)$]
	\item Letting $Y = \{x \in X : \exists a \in G_1 \text{ such that } \pi(a) \in \mathbb{Q}^+x\}$, there is $1$-to-$1$ map $h: Y \rightarrow X$ such that $y \in Y \Rightarrow \pi(h(y)) \in \mathbb{Q}^+y$.
\end{enumerate}
[Why? By $(*_2)$ and $(*_3)$.]
\begin{enumerate}[$(*_5)$]
	\item $Y = X$.
\end{enumerate}
[Why? Because $\pi$ is onto, by $(*_2)$.]
\begin{enumerate}[$(*_6)$]
	\item If $y \in X \setminus \mrm{ran}(h)$, then $\pi(y) = 0$.
\end{enumerate}
Why $(*_6)$? Toward contradiction, assume $y \in X \setminus \mrm{ran}(h)$ and $\pi(y) \neq 0$. Let $\pi(y) = \sum_{\ell < k} q_\ell x_\ell$, $k \geq 1$, $q_\ell \in \mathbb{Q}^+$, $(x_\ell : \ell < k) \in \mrm{seq}_k(X)$. Let also, for $\ell < k$, $z_\ell = h(x_\ell)$. Clearly $(y)^\frown(z_\ell : \ell < k)$ is without repetitions, as $y \notin \mrm{ran}(h)$, $(x_\ell : \ell < k) \in \mrm{seq}_k(X)$ and $h$ is $1$-to-$1$. Notice also that by $(*_4)$, $\pi(z_\ell) = r_\ell x_\ell$, for some $r_\ell \in \mathbb{Q}^+$.
 Let now $q'_\ell = q_\ell/r_\ell$, if $\ell \in \{1, ..., k-1\}$, and $q'_0 \in \mathbb{Q} \setminus (\{q_0/r_0\} \cup \{0\})$. Let then $m > 1$ be such that $mq, mq_\ell, m q'_\ell \in \mathbb{Z}$ and consider:
	\begin{equation}
	\tag{$*_{6.1}$} a = my + \sum_{\ell < k} mq'_\ell z_\ell \in G^+_0.
\end{equation}
Then $a \in G_0 \subseteq G_1$ and we have:
$$\begin{array}{rcl}
\pi(a)    & = & \pi(my + \sum_{\ell < k} mq'_\ell z_\ell) \\
		& = & m\pi(y) + \sum_{\ell < k} mq'_\ell r_\ell x_\ell \\
		& = & m \sum_{\ell < k} q_\ell x_\ell + \sum_{\ell < k} mq'_\ell r_\ell x_\ell \\
		& = & m\sum_{\ell < k} (q_\ell - q'_\ell r_\ell) x_\ell \\
	    & = & m(q_0 - q'_0r_0) x_0.
\end{array}$$
As $q'_0 \neq q_0r_0$,  $\mrm{supp}(\pi(a))$ is a singleton, but $\mrm{supp}(a)$ is not a singleton, as $y, z_0 \in \mrm{supp}(a)$, recalling that $q'_0 \neq 0$ and $z \neq y$. But this contradicts $(*_2)$.
\begin{enumerate}[$(*_7)$]
	\item $X = \mrm{ran}(h)$.
\end{enumerate}
Why $(*_7)$? If not, let $x \in X \setminus \mrm{ran}(h)$, so $\pi(x) = 0$. Let now $y = h(x)$ and then let:
\begin{enumerate}[$(*_{7.1})$]
	\item $\pi(y) = \pi(h(x)) = rx$, for some $r \in \mathbb{Q}^+$, and let $p = p_{x+y}$.
\end{enumerate}
Then $\pi(x+y) \in \mathbb{Q}^+x$, as $\pi(x+y) = \pi(x) + \pi(y) = 0 + rx = rx$, recalling $(*_{7.1})$. But this contradicts $(*_2)$. This concludes the proof of $(*_{7})$.
\begin{enumerate}[$(*_{8})$]
	\item If $x \in X$, then $\pi(x) = q_x h^{-1}(x)$, where $q_x \in \mathbb{Q}^+$.
\end{enumerate}
[Why? Given $x \in X$, by $(*_7)$ we can choose $y \in X$ such that $h(y) = x$, hence by the choice of $h$ there is $q_*$ s.t. $\pi(h(y)) = q_xy$, which means $\pi(y) = q_x y = q_x h^{-1}(x)$.]
\begin{enumerate}[$(*_9)$]
	\item 
	\begin{enumerate}
	\item $g = h^{-1}$ is a permutation of $X$;
	\item for some $q_* \in \mathbb{Q}$, $x \in X \Rightarrow \pi(x) = q_*g(x)$;
	\item $q_* \in \mathbb{Z}^+$;
	\item if $y = g(x)$, then $y \in x/E_1$.
\end{enumerate}
\end{enumerate}
Why $(*_9)$? Clause (a) holds as $h: Y \rightarrow X$ is $1$-to-$1$ by $(*_4)$, $Y = X$ by $(*_5)$ and $h$ is onto $Y$ by $(*_7)$. If clause (b) fails, then for some $x_1 \neq x_2 \in X$ by $(*_{8})$ we have $\pi(x_\ell) = q_\ell g(x_\ell)$, for $\ell = 1, 2$ and $q_1 \neq q_2 \in \mathbb{Q}$.
Let $p = p_{x_1 - x_2}$, so $x_1 - x_2 \in G_{(1, p)}$ and thus:
	$$\pi(x_1 - x_2) = q_1g(x_1) - q_2g(x_2) \in G_{(1, p)}.$$
But $q_1-q_2 \neq 0$ and this contradicts $(*_{10})$ below. Concerning clause (c) suppose that $q_* \notin \mathbb{Z}$, then necessarily for some prime $p$ and $m \geq 1$ we have that $\frac{1}{p} = mq_*$. As $(p_a : a \in G^+_0)$ is with no repetitions, there is $x \in X$ such that $p_x \neq p$, hence $\frac{1}{p}x \in G_2 \setminus G_1$ (recalling \ref{def_G02_cohopfian}  and \ref{def_G1_cohopfian}), hence $mq_*x \notin G_1$, but $mq_*x = \pi(mx)$ and $mx \in G_1$, a contradiction. Finally, we prove clause (d). We have $\pi(x) \in \mathbb{Q}^+g(x)$, so by \ref{lemma_pre_main_th}(\ref{G1p_pure_co-hop}) there are $j_* < \omega$ and, for $j < j_*$, $y_{j} \in X$ and $q_{j} \in \mathbb{Q}^+$ such that $y_{j} \in x_i/\mathcal{E}_{\mathfrak{m}}$ and $\pi(x) = \sum_{j < j_*} q_{j} y_{j}$. Clearly, $j_* > 0$, as $\pi$ is an automorphism, in fact $j_* = 1$ and so $g(x) = y_{0}$, as $y_{0} \in x/E_1$. Thus $g(x) \in x/E_1$. 
This proves $(*_9)$.
\begin{enumerate}[$(*_{10})$]
	\item If $x_1 \neq x_2 \in X$, $p = p_{x_1 - x_2}$ and $\sum_{\ell < k} q_\ell y_\ell \in G_{(1, p)}$, then $\sum_{\ell < k} q_\ell = 0$.
\end{enumerate}
Why $(*_{10})$? By \ref{lemma_pre_main_th}(\ref{G1p_pure_co-hop}) and the choice of $p$, for some $m \in \mathbb{Z}^+$, $j < \omega$, $q^*_i \in \mathbb{Q}^+$ and $(y_{(i, 1)}, y_{(i, 2)}) \in (x_1, x_2)/E_2$ we have:
\begin{enumerate}[$(*_{10.1})$]
	\item $m(\sum_{\ell < k} q^*_\ell y_\ell) = \sum_{i < j} q_i(y_{(i, 1)} - y_{(i, 2)})$.
\end{enumerate}
Hence, letting $Y = \{y_{(i, 1)}, y_{(i, 2)} : i < j\}$, we have:
\begin{enumerate}[$(*_{10.2})$]
	\item $m(\sum_{\ell < k} q_\ell y_\ell) = \sum \{r_y y : y \in Y\}$, where we let:
	$$r_y = \sum \{q_i : i < j, y_{(i, 1)} = y \} - \sum \{q_i : i < j, y_{(i, 2)} = y \}.$$
\end{enumerate}
Hence, we have:
\begin{enumerate}[$(*_{10.3})$]
	\item 
	$$\begin{array}{rcl}
\sum_{\ell < k} m q_\ell & = & \sum \{r_y : y \in Y\} \\
		& = & \sum_{y \in Y} (\sum\{q_i : i < j, y_{(i, 1)} = y \} - \sum \{q_i : i < j, y_{(i, 2)} = y \}) \\
		& = & \sum \{q_i : i < j\} - \sum \{q_i : i < j\} = 0. 
\end{array}$$
\end{enumerate}
And, as $m \neq 0$, we are done. This concludes the proof of $(*_{10})$.

\smallskip
\noindent Now we continue similarly to the proof of \cite[Theorem~5.16]{1205}, specifically we claim:
\begin{enumerate}[$(*_{11})$]
	\item $g$ is the identity function on $X$.
\end{enumerate}
\underline{Case 1}. The set $Y = \{x/E_1 : \text{ for some } y \in x/E_1, \; g(y) \neq y\}$ is infinite.
\begin{enumerate}[$(\star_{1})$]
\item Choose $x_i$, $n_i$, for $i < \omega$, such that:
\begin{enumerate}
	\item $n_i$ is increasing with $i$;
	\item $x_i \in Y \cap X_{n_{i+1}} \setminus X_{n_i}$;
	\item $g(x_i)  \neq x_i$, $g(x_i) \in X_{n_{i + 1}}$, $g(x_i) \notin X_{n_i}$, $(x_i/E_1) \cap X_{n_i} = \emptyset$;
	\item $(x_i/E_1 : i < \omega)$ are pairwise distinct (this actually follows);
	\item there is $\bar{t} \subseteq T_{n_{i+1}}$ and $\eta$ s.t. $f_{(\bar{t}, \eta)}(x_i) = g(x_i)$.
\end{enumerate}
\end{enumerate}
[Why we can do this? By the case assumption and $(*_8)$, in particular $(*_8)$(d).]
\newline Note now that, for $i < \omega$, we have:
\begin{enumerate}[$(\star_{2})$]
	\item $g(x_i)\in x_i/E_1$.
\end{enumerate}
[Why $(\star_{2})$? By $(*_9)$(d).]
\begin{enumerate}[$(\star_{3})$]
	\item The sequence $(x_i, g(x_i) : i < \omega)$ is with no repetitions.
\end{enumerate}
Why $(\star_{3})$? $x_i, g(x_i) \in X$, $g(x_i) \neq x_i$, by the choice of $x_i$. Also, if $i < j$, then:
	\begin{enumerate}[$(\cdot_1)$]
	\item $x_i/E_1 \neq x_j/E_1$ (by $(\star_{1})$(d));
	\item $g(x_i) \in x_i/E_1$ (by $(\star_{2})$);
	\item $g(x_j) \in x_j/E_1$ (by $(\star_{2})$).
	\end{enumerate}
Together we are done proving $(\star_{3})$.

\smallskip
\noindent
\begin{enumerate}[$(\star_{4})$]
\item For each $2 \leq r < \omega$, let:
$$x^+_r = \sum_{\ell \leq r} x_\ell, \;\;\;\;\; p_r = p^+_{x_r}, \;\;\;\;\; \bar{x}_r = (x_\ell : \ell \leq r) \in \mrm{seq}_{r+1}(X).$$
\end{enumerate}

\smallskip
\noindent
As $\pi \in \mrm{End}(G_1)$, clearly $\pi(x^+_r) \in G_{(1, p_r)}$, hence by \ref{lemma_pre_main_th}(\ref{G1p_pure_co-hop}) applied to $x_r$, $p_{r}$ here standing for $a$, $p_a$ there we can find $j_r, m_r >0$, and, for $j < j_r$, $\bar{y}^{(r, j)} \in \mrm{seq}_{r+1}(X)$, $q^r_j \in \mathbb{Q}^+$, $b^r_j \in G^+_1$ such that the following holds:
\begin{enumerate}[$(\star_{5})$]
	\item 
	\begin{enumerate}[(a)]
	\item for $j < j_r$, $\bar{y}^{(r, j)} \in \bar{x}_r/E_{r+1}$;
	\item $(b^r_j = \sum_{\ell \leq r} y^{(r, j)}_\ell : j < j_r)$ is linearly independent;
	\item $\pi(x^+_r) = \sum \{q^r_j b^r_j : j < j_r \}$;
	\item for $j < j_r$, $m_r x^+_r \mathcal{E}_{\mathfrak{m}} b^r_j$ (and $m_r x^\star_r \in G_0$, $m_r b^r_j \in G_0$).
\end{enumerate}
\end{enumerate}
But by $(*_8)$ we have:
	\begin{enumerate}[$(\star_{6})$]
	\item $\pi(x^+_r) = \pi(\sum_{\ell \leq r} x_\ell) = \sum_{\ell \leq r} \pi(x_\ell) = \sum_{\ell \leq r} q^*g(x_\ell)$.
	\end{enumerate}
Hence, by $(\star_{5})$,  noticing that $(x_\ell/E_1 : \ell \leq r)$ is with no repetitions and that $y^{(r, j)}_\ell \in x_\ell/E_1$,
we have:
	\begin{enumerate}[$(\star_{7})$]
	\item $j_r = 1$.
	\end{enumerate}
	\begin{enumerate}[$(\star_{8})$]
	\item Let $(\bar{t}_r, \eta_r) \in T^{n_r} \times \{1, -1\}^{n_r}$ be such that:
	$$f_{(\bar{t}_r, \eta_r)}(\bar{x}_r) = \bar{y}^{(r, 0)} \; \text{ and  w.l.o.g. $n_r$ is minimal}.$$
	\end{enumerate}
Hence, we have:
\begin{enumerate}[$(\star_{9})$]
	\item If $\ell \leq r$, then $f_{(\bar{t}_r, \eta_r)}(x_\ell) = g(x_\ell)$.
	\end{enumerate}
Now, applying \ref{claim_for_star}(\ref{claim_for_starE}), to $(\bar{t}, \eta_r, x_r)$ we get:
\begin{enumerate}[$(\star_{10})$]
	\item There are $(\bar{t}'_r, \eta'_r)$ such that they satisfy:
	\begin{enumerate}
	\item $f_{(\bar{t}'_r, \eta'_r)}(x_r) = g(x_r)$;
	\item $(\bar{t}'_r, \eta'_r)$ is strongly reduced for $x_r$.
	\end{enumerate}
	\end{enumerate}
	Recall that, by $(*_{11})$(c), $(X_r/E_1) \cap X_{n_r} = \emptyset$, hence by \ref{claim_for_star}(\ref{claim_for_starG}) and $(*_{10})$(a), $X_{n_r} \subseteq \mrm{dom}(f_{(\bar{t}'_r, \eta'_r)})$, and by \ref{claim_for_star}(\ref{claim_for_starH}), $f_{(\bar{t}'_r, \eta'_r)} \restriction X_{n_r} = f_{(\bar{t}_r, \eta_r)} \restriction X_{n_r}$, so $\{x_i : i \leq r \} \subseteq \mrm{dom}(f_{(\bar{t}'_r, \eta'_r)})$, and thus we have:
	\begin{enumerate}[$(\star_{11})$]
		\item $f_{(\bar{t}'_r, \eta'_r)}$ and $f_{(\bar{t}_r, \eta_r)}$ agree on $\{x_\ell : \ell \leq r\}$.
	\end{enumerate}
	Recall now that $n_r$ was chosen to be minimal (cf. $(\star_{8})$) and so by $(\star_{11})$ we have that $n_r = \mrm{lg}(\bar{t}'_r)$. Hence, by \ref{claim_for_star}(\ref{claim_for_starH}), we have the following:
	\begin{enumerate}[$(\star_{12})$]
		\item $(\bar{t}'_r, \eta'_r) = (\bar{t}_r, \eta_r)$, so $(\bar{t}_r, \eta_r)$ is strongly reduced for $x_r$.
	\end{enumerate}
	Further, we have:
	\begin{enumerate}[$(\star_{14})$]
		\item if $r < s < \omega$, then for all $\ell < n_r$ there is $j < n_s$ such that:
$$t_{(r, \ell)} \leq_T t_{(s, j)} \text{ and } t_{(s, \ell)} \notin X_{n_r},$$
where we let $\bar{t}_r = (t_{(r, \ell)} : \ell < n_r)$ and $\bar{t}_s = (t_{(r, \ell)} : \ell < n_s)$.
	\end{enumerate}
	[Why? By \ref{claim_for_star}(\ref{claim_for_starH}): as $(\bar{t}_r, \eta_r)$ is strongly reduced for $x_r$, $f_{(\bar{t}_r, \eta_r)} = g(x_r) = f_{(\bar{t}_s, \eta_s)}$ and $x_r /E_1 \cap X_{n_r} = \emptyset$.]
	\begin{enumerate}[$(\star_{15})$]
	\item For $r < \omega$, we can choose $\ell_r < n_r$ such that we have:
	$$t_{(r, \ell_r)} \leq_T t_{(r+1, \ell_{r+1})} \text{ and } t_{(r, \ell_r)} \notin X_{n_{r}}.$$
\end{enumerate}
	[Why? Choose $\ell_r$ by induction on $r$ using $(\star_{14})$.]
\newline This implies that $T$ has an infinite branch, thus concluding the proof of Case 1.
	
\medskip
\noindent
\underline{Case 2}. The set $Y = \{x/E_1 : \text{ for some } y \in x/E_1, g(y) \neq y\}$ is finite and $\neq \emptyset$.
\newline Notice that by $(*_8)$ we have that $\pi(x) \notin \mathbb{Q}^+x$ iff $g(x) \neq x$. Now, choose $x_0 \in X$ such that $\pi(x_0) \notin \mathbb{Q}x_0$.
\noindent Let $n < \omega$ be such that $x_0 \in X_n$ and choose $x_1$ such that:
	\begin{enumerate}[$(\oplus_1)$]
	\item 
	\begin{enumerate}
	\item $x_1 \in X \setminus \bigcup\{y/E_1 : y \in X_n\}$;
	\item $\pi(x_1) \in \mathbb{Q}x_1$.
	\end{enumerate}
\end{enumerate}
[Why possible? By the assumption in Case 2 and \ref{hyp_A2}(\ref{itemh_co}).]
\newline Notice now that:
\begin{enumerate}[$(\oplus_2)$]
	\item For $\ell = 0, 1$ we have $\mrm{supp}(\pi(x_\ell)) \subseteq x_\ell/E_1$.
\end{enumerate}
\begin{enumerate}[$(\oplus_3)$]
	\item By \ref{lemma_pre_main_th}, there are $(\bar{t}_j, \eta_j, q_j : j < j_*)$ such that:
	\begin{enumerate}[(a)]
	\item $j_* < \omega$, $(\bar{t}_j : j < j_*)$ is injective and, for $j < j_*$, $q_j \in \mathbb{Q}^+$ and $\eta_j \in \{1, -1\}^{\mrm{lg}(\bar{t}_j)}$;
	\item $\pi(x_0 + x_1) = \sum \{q_j f_{(\bar{t}_j, \eta_j)}(x_0 + x_1) : j < j_*\}$.
	\end{enumerate}
\end{enumerate}
\begin{enumerate}[$(\oplus_4)$]
	\item As in $(\star_5)$ in Case 1, we have:
	\begin{enumerate}[(a)]
	\item $\pi(x_0) = \sum \{q_j f_{(\bar{t}_j, \eta_j)}(x_0) : j < j_*\}$;
	\item $\pi(x_1) = \sum \{q_j f_{(\bar{t}_j, \eta_j)}(x_1) : j < j_*\}$.
	\end{enumerate}
\end{enumerate}
\begin{enumerate}[$(\oplus_5)$]
	\item $\pi(x_0 + x_1) = \pi(x_0) + \pi(x_1) = q_*g(x_0) + q_*g(x_1) = q_*g(x_0) + q_*x_1$
\end{enumerate}
By $(\oplus_4)$+$(\oplus_5)$, recalling $x_1 \notin x_0/E_1$, clearly $j_* = 1$, but by $(\oplus_4)$(b)+$(\oplus_5)$ we have $f_{(\bar{t}_0, \eta_0)}(x_1) = x_1$, hence $\bar{t}_0 = ()$ and so $f_{(\bar{t}_0, \eta_0)}(x_0) = x_0$, leading to a contradiction.

\smallskip
\noindent
\underline{Case 3}. The set $Y = \{x/E_1 : \text{ for some } y \in x/E_1, \pi(y) \notin \mathbb{Q}y\}$ is empty.
\newline As Case 1 and Case 2 are not possible, we are left with Case 3, so $(*_{11})$ holds.
\begin{enumerate}[$(*_{12})$]
	\item $q_* \in \{1, -1\}$.
\end{enumerate}
Why $(*_{12})$? If not, recalling $q_* \in \mathbb{Z}^+$ by $(*_8)(c)$, let $p$ be a prime dividing $q_*$. 
\begin{enumerate}[$(*_{12.1})$]
	\item There is $x \in X$ such that $G_1 \models p \not \vert \; x$.
\end{enumerate} 
[Why? If $p \not\in \{p_a : a \in G^+_0\}$, then by \ref{def_G1_cohopfian}(\ref{def_G1}) any $x \in X$ is OK. Suppose then that $p = a$ and let $Y = \{y/E_1 : y \in \mrm{supp}(a)\}$. As $\mrm{supp}(a)$ is finite, clearly $Y \neq X$ and so it suffices to choose any $x \in X \setminus Y$ and then recall \ref{def_G1_cohopfian}(\ref{def_G1}).]
\newline As $\pi \in \mrm{End}(G_1)$ is onto, there is $b \in G_1$ such that $\pi(b) = x$. By $(*_{11})$ necessarily $b \in \mathbb{Q}^+x$, so let $b = q_0x \in G_1$. Then we have:
	$$\pi(b) = \pi(q_0x) = q_*q_0 x = x,$$
so $q_0q_* = 1$. But, by $(*_9)$(c), $q_* \in \mathbb{Z}^+$ and $p$ divides $q_*$, so necessarily we have that:
$$q_0 \notin \{\frac{m}{n} : m \in \mathbb{Z}, n \geq 1, p \not \vert \, n\}.$$
But in $G_1$ we have that $x$ is not divisible by $p$, a contradiction.
\begin{enumerate}[$(*_{13})$]
	\item $\pi \in \{ \mrm{id}_{G_1}, -\mrm{id}_{G_1} \}$.
\end{enumerate}
[Why? By $(*_{9})$, $(*_{11})$ and $(*_{12})$.]
\end{proof}

	\begin{proof}[Proof of Theorem~\ref{main_th4}] By \ref{main_th_Sec1}.
\end{proof}

\section{The Hopfian problem for countable $\mrm{TFAB}$}\label{sec_hop}

\begin{hypothesis}\label{hyp_A1_ho} Throughout this section the following hypothesis stands:
	\begin{enumerate}[(1)]
	\item\label{omega_levels} $T = (T, <_T)$ is a rooted tree with $\omega$ levels;
	\item we denote by $\mathrm{lev}(t)$ the level of~$t$;
	\item $T = \bigcup_{n < \omega} T_n$, $T_n \subseteq T_{n+1}$, $|T_{n+1} \setminus T_n| \geq 2$, and $t \in T_n$ implies that $\mathrm{lev}(t) \leq n$;
	\item $T_0 = \{t_0 \}$, where $t_0$ is the root of $T$, and, for $n < \omega$, $T_n$ is finite; 
	\item $T_{<n} = \bigcup_{\ell < n} T_\ell$ (so $T_{<(n+1)} = T_n$);
	\item if $s <_T t \in T_n$, then $s \in T_{< n}$;
	\item for $t \in T$, we let $\mathbf{n}(t)$ be the unique $0 \leq n < \omega$ such that $t \in T_n \setminus \bigcup_{\ell < n} T_{\ell}$.
\end{enumerate}
\end{hypothesis}

	\begin{definition}\label{hyp_A2_ho} Let $\mathrm{K}^{\mrm{ho}}(T)$ be the class of objects:
	$$\mathfrak{m}(T) = \mathfrak{m} = (X^T_n, \bar{f}^T_n: n < \omega) = (X_n, \bar{f}_n : n < \omega)$$
satisfying the following requirements:
	\begin{enumerate}[(a)]
	\item\label{X0_not_empty_ho} $X_0 \neq \emptyset$, $|X_n| < \aleph_0$, $X_n \subsetneq X_{n+1}$, $X_{< n} = \bigcup_{\ell < n} X_\ell$ and $X = \bigcup_{n < \omega} X_n$;
	\item $\bar{f} = (f_t : t \in T)$ and $\bar{g} = (g_t : t \in T)$;
	\item\label{def_of_f_t_ho} if $n >0$ and $t \in T_n \setminus T_{< n}$, then $f_t$ is a partial functions from $X_{n}$ onto $X_{n-1}$ and $f_t$ maps $\mrm{dom}(f_t) \setminus X_{n-1}$ onto $X_{n-1}$;
	\item if $s <_T t \in T_n$, then $f_s \subseteq f_t$; 
    \item if $s <_T t$ and $s \in T_{n+1} \setminus T_n$, then $\mrm{dom}(f_t) \cap X_{n+1} = \mrm{dom}(f_s)$;
	\item\label{item_g_hopfian_tree} if $t \in T_n \setminus T_{< n}$, then $\mrm{dom}(f_t) \cap X_{n-1} = \bigcup \{\mrm{dom}(f_s) : s <_T t\}$;
	\item\label{item_h_hopfian_tree} if $s, t \in T_n$ and $x \in \mrm{dom}(f_s) \cap \mrm{dom}(f_t)$, then for some $r \in T_n$ such that $r \leq_T s, t$ we have that $x \in \mrm{dom}(f_r)$, equivalently, $\mrm{dom}(f_s) \cap \mrm{dom}(f_t) = \mrm{dom}(f_r)$, for $r = s \wedge t$, where $\wedge$ is the natural semi-lattice operation taken in the tree.
	\item\label{hyp_A2_ho_mapped_to_0} $X_{n+1} \supsetneq X_n \cup \bigcup \{\mrm{dom}(f_t) : t \in T_{n+1} \setminus T_n\}$;
	\item for every $x$ and $y$, $\{t \in T : f_t(x) = y \}$ is a cone of $T$;
	\item for every $s, t \in T$, $\{x \in X : f_s(x) = f_t(x) \}$ is $\mrm{dom}(f_{s \wedge t})$;
	\item\label{thezs} for $n < \omega$ and $t \in T_n \setminus T_{< n}$, letting $Z_n = \bigcup_{k \leq |X_{< n}|} \mrm{seq}_k(X_{< n})$ (so a finite set), we can find $(\bar{z}_{(n, t, \bar{x})} : \bar{x} \in \mrm{seq}(Z_n))$ such that the following conditions are met:
	\begin{enumerate}[$(\cdot_1)$]
	\item $\bar{z}_{(n, t, \bar{x})} \in \mrm{seq}_{\mrm{lg}(\bar{x})}(\mrm{dom}(f_t) \setminus X_{< n})$; 
	\item $(\mrm{ran}(\bar{z}_{(n, t, \bar{x})}) : \bar{x} \in \mrm{seq}(Z_n), t \in T_n \setminus T_{n-1})$ is a seq. \mbox{of pairwise disjoint sets;}
	\item $f_t(\bar{z}_{(n, t, \bar{x})}) = \bar{x}$.
	\end{enumerate}
\end{enumerate}
\end{definition}

	\begin{remark} Notice that \ref{hyp_A2_ho}(\ref{thezs}) requires $X_n \setminus X_{n-1}$ to be large enough.
\end{remark}

	\begin{cclaim} For $T$ as in Hypothesis~\ref{hyp_A1_ho}, $\mathrm{K}^{\mrm{ho}}(T) \neq \emptyset$ (cf. Definition~\ref{hyp_A2_ho}).
\end{cclaim}

\begin{convention}\label{the_m_convention_ho} $\mathfrak{m} = (X_n, \bar{f}_n : n < \omega) \in \mathrm{K}^{\mrm{ho}}(T)$.
\end{convention}

\begin{definition}\label{def_<X_ho} On $X$ (cf. Convention~\ref{the_m_convention_ho}) we define:
	\begin{enumerate}[(1)]
	\item for $\bar{x} \in \mrm{seq}_k(X)$, we let:
	$$\mrm{suc}(\bar{x}) = \{\bar{y} \in \mrm{seq}_k(X) : f_t(\bar{y}) = \bar{x}, \text{ for some } t \in T \text{ s.t. } \bar{x} \subseteq \mrm{dom}(f_t) \},$$
	\item we let $\bar{x} \leq^k_\mathfrak{m} \bar{y}$ if $\bar{x}, \bar{y} \in \mrm{seq}_k(X)$ and there is a sequence $(\bar{x}_\ell : \ell \leq n)$ such that $\bar{x}_0 = \bar{x}$, $\bar{x}_n = \bar{y}$ and for every $\ell < n$, $\bar{x}_{\ell+1} \in \mrm{suc}(\bar{x}_\ell)$;
	\item If $k = 1$, we may simply write $\leq_{\mathfrak{m}}$ instead of $\leq^1_{\mathfrak{m}}$.
	\end{enumerate}
\end{definition}

	\begin{cclaim}\label{observation_Xtree_ho}
	\begin{enumerate}[(1)]
	\item $(X, <_{\mathfrak{m}})$ is a tree with $\omega$ levels;
	\item if $z \in X_0$, then $z$ is a root of the tree $(X, <_{\mathfrak{m}})$;
	\item if $x \in X_{n+1} \setminus X_n$, then $x$ is of level $\leq n+1$ and for at most one $y \in X_n$ we have that $x \in \mrm{suc}(y)$;
	\item if $f_{t_1}(x)$ and $f_{t_2}(x)$ are well-defined, then they are equal;
	\item $(\mrm{seq}_k(X), <^k_{\mathfrak{m}})$ is a tree with $\omega$ levels;
	\item if $\bar{x} \leq^k_{\mathfrak{m}} \bar{y}$, then there is exactly one pair $(\mathbf{x}, \bar{t})$ such that for a unique $n$ we have:
	\begin{enumerate}
	\item $\mathbf{x} = (\bar{x}_\ell : \ell \leq n)$, $\bar{x}_\ell \in \mrm{seq}_k(X)$ and $(\bar{x}_0, \bar{x}_n) = (\bar{x}, \bar{y})$;
	\item $\bar{t} = (t_\ell : \ell < n)$, $t_\ell \in T$;
	\item $f_{t_\ell}(\bar{x}_{\ell+1}) = \bar{x}_\ell$, for $\ell < n$;
	\item if $t < t_\ell$, then $f_t(\bar{x}_{\ell+1})$ is not well-defined.
	\end{enumerate}
	\item\label{observation_Xtree_ho7} For every $n$ there is $x \in X_{n+1} \setminus X_n$ which is $<_{\mathfrak{m}}$-minimal.
	\end{enumerate}
\end{cclaim}

	\begin{proof} We only prove (3). Assume that $x \in \mrm{suc}(y_1) \cap \mrm{suc}(y_2)$. So for $\ell \in \{1, 2\}$, there is $t_\ell \in T$ such that $f_{t_\ell}(x) = y_\ell$, and w.l.o.g. $s <_T t_\ell$ implies that $x \not \in \mrm{dom}(f_s)$.  Now, by assumption $x \in X_{n+1} \setminus X_n$, and so $t_1, t_2 \in T_n \setminus T_{< n}$ (by Definition~\ref{hyp_A2_ho}(\ref{item_g_hopfian_tree})). Now, by Definition~\ref{hyp_A2_ho}(\ref{item_h_hopfian_tree}), we have:
	$$t_1 \neq t_2 \Rightarrow \mrm{dom}(f_{t_1}) \cap \mrm{dom}(f_{t_2}) \subseteq X_{<n},$$
hence $t_1 = t_2$, which implies that $y_1 = f_{t_1}(x) = f_{t_2}(x) = y_2$, as promised.
\end{proof}


	\begin{definition}\label{def_G02_hopfian} Let $\mathfrak{m} \in \mathrm{K}^{\mrm{ho}}(T)$ (i.e. as in Convention~\ref{the_m_convention}).
	\begin{enumerate}[(1)]
	\item Let $G_2 = G_2[\mathfrak{m}]$ be $\bigoplus \{ \mathbb{Q}x : x \in X\}$.
	\item Let $G_0 = G_0[\mathfrak{m}]$ be the subgroup of $G_2$ generated by $X$, i.e. $\bigoplus \{ \mathbb{Z}x : x \in X\}$.
	\item For $t \in T_n \setminus T_{< n}$, let:
	\begin{enumerate}[(a)]
	\item $H_{(2, t)} = \bigoplus \{ \mathbb{Q}x : x \in X_n\}$;
	\item $I_{(2, t)} = \bigoplus \{ \mathbb{Q}x : x \in X_{n-1} \}$;
	\item\label{hatf_on_G2_ho} $\hat{f}_t$ is the (unique) homomorphism from $H_{(2, t)}$ onto $I_{(2, t)}$ such that $x \in X_n \setminus \mrm{dom}(f_t) \Rightarrow \hat{f}_t(x) = 0$ and $x \in \mrm{dom}(f_t) \Rightarrow \hat{f}_t(x) = f_t(x)$ (cf. \ref{hyp_A2_ho}(\ref{def_of_f_t_ho})).
	\end{enumerate}
	\item For $t \in T$, we define $H_{(0, t)} := H_{(0, t)} \cap G_0$ and $I_{(0, t)} := I_{(2, t)} \cap G_0$.
	\item For $\hat{f}_t$ as above, we have that $\hat{f}_t [H_{(0, t)}] = I_{(0, t)}$, consequently when we talk about $\hat{f}_t$ as a map on $G_0$ we actually mean $\hat{f}_t \restriction H_{(0, t)}$.
	\item\label{def_<*_ho} We define the partial order $<_*$ on $G^+_0$ by letting $a <_* b$ if and only if $a \neq b \in G^+_0$ and, for some $n < \omega$,  $a_0, ..., a_n \in G_0, a_0 = a, a_n = b$ and:
	$$\ell < n \Rightarrow \exists t \in T(\hat{f}_t(a_{\ell + 1}) = a_{\ell}).$$
\end{enumerate}
\end{definition}

	\begin{cclaim}\label{G0_is_tree_ho} \begin{enumerate}[(1)]
	\item $(G_0, <_*)$ is a tree with $\omega$ levels (recall Hypothesis~\ref{hyp_A1_ho}(\ref{omega_levels})).
	\item If $s <_T t \in T$, then $\hat{f}_s \subseteq \hat{f}_t$, moreover if $s \in X_m \setminus X_{< m}$, then $f_s = f_t \restriction X_m$. 
	\item If $a \in G^+_0$, then:
	\begin{enumerate}[(a)]
	\item $a$ has a unique representation $\sum_{\ell < n} q_\ell x_\ell$, for some $n \geq 1$, $q_\ell \in \mathbb{Q}^+$, $(x_\ell : \ell < n) \in \mrm{seq}_n(X)$;
	\item the sequence $((q_\ell, x_\ell) : \ell < n)$ is unique up to permutation;
	\item $a \in \mrm{dom}(\hat{f}_t)$ iff $\ell < n \Rightarrow x_\ell \in X_{\mathbf{n}(t)}$;
	\item if $b = \hat{f}_t(a)$, then for some $u \subseteq n$ we have $b = \sum \{q_\ell f_t(x_\ell) : \ell \in u\}$, $u = \{ \ell < n : x_\ell \in \mrm{dom}(f_t)\}$.
	\end{enumerate}
	\item If $a, b \in G^+_0$ and $a \leq_* b$, then there is a unique pair $(\bar{a}, \bar{t})$ such that for some unique $n < \omega$ we have:
	\begin{enumerate}
	\item $\bar{a} = (a_\ell : \ell \leq n)$;
	\item $\bar{t} = (t_\ell : \ell < n)$, $t_\ell \in T$;
	\item $\hat{f}_{t_\ell}(a_{\ell+1}) = a_\ell$;
	\item if $t < t_\ell$, then $a_{\ell+1} \notin \mrm{dom}(\hat{f}_t)$.
	\item $a_n = b$ and $a_0 = a$.
	\end{enumerate}
	\end{enumerate}
\end{cclaim}

	\begin{proof} Straightforward. 
\end{proof}

	\begin{cclaim}\label{technical_claim_on_order}
	\begin{enumerate}[(1)]
	\item Let $Y = \bigcup_{t \in T} \{ \mrm{dom}(f_t) \}$. There is $\mathbf{h}: Y \rightarrow T$ s.t. for $x \in Y$:
	\begin{enumerate}[(a)]
	\item there is $\mathbf{n}(x) > 0$ such that $x \in X_{\mathbf{n}(x)} \setminus X_{< \mathbf{n}(x)}$ and $\mathbf{n}(\mathbf{h}(x)) = \mathbf{n}(x)$;
	\item $x \in \mrm{dom}(f_{\mathbf{h}(x)})$;
	\item for $t \in T$, $x \in \mrm{dom}(f_t) \Leftrightarrow \mathbf{h}(x) \leq_{T} t$;
	\item if $f_t(x)$ is well-defined, then $f_t(x) = f_{\mathbf{h}(x)}(x) \in X_{<\mathbf{n}(x)}$.
	\end{enumerate}
	\item $(X, <_\mathfrak{m})$ is not well-founded.
	\end{enumerate}
\end{cclaim}

	\begin{definition}\label{def_G1_hopfian} Let $(p_a : a \in G^+_0)$ be a sequence of pairwise distinct primes s.t.:
$$a = \sum_{\ell < k} q_\ell x_\ell, \, q_\ell \in \mathbb{Z}^+, \, (x_\ell : \ell < k) \in \mrm{seq}_k(X) \Rightarrow p_a \not \vert \; q_\ell.$$ 
	\begin{enumerate}[(1)] 
	\item\label{P_a} For $a \in G^+_0$, let $\mathbb{P}_a = \bigcup_{n < \omega} \mathbb{P}_{(a, n)}$ and $G_1 = G_1[\mathfrak{m}] = G_1[\mathfrak{m}(T)] = G_1[T] = \bigcup_{n < \omega} G_{(1, n)}$, where, by induction on $n < \omega$, we define $G_{(1, n)}$, $\mathbb{P}_{(a, n)}$ as follows:
	\begin{enumerate}[(a)]
	\item $G_{(1, n)}$ is a subgroup of $G_2$ extending $G_0$;
	\item $G_{(1, n)}$ is increasing with $n < \omega$;
	\item 
	\begin{enumerate}[(i)]
	\item for $a \in G^+_0$ and $n < \omega$, we shall define a set of powers of primes $\mathbb{P}_{(a, n)}$ increasing with $n$;
	\item $G_{(1, n)}$ is the subgroup of $G_2$ generated by the following set:
	$$G_0 \cup \{p^{-m}a : a \in G^+_0 \text{ and } p^m \in \mathbb{P}_{(a, n)}\};$$
	\item for $n = 0$, $\mathbb{P}_{(a, n)}$ is the set of $p^m$ such that there is an element $b \in G^+_0$ such that $p = p_b$, $a \leq_* b$ and $m < \omega$;
	\end{enumerate}
	\item\label{itemd} if $n = j+1$, then $\mathbb{P}_{(a, n)}$ is the union of $\mathbb{P}_{(a, j)}$ together with the $p^m$ such that for some $t, k, \bar{x}, \bar{q}, \bar{y}$ we have:
	\begin{enumerate}[(i)]
	\item $t \in T_n \setminus T_{< n}$;
	\item $0 < k \leq |X_j| = |X_{<n}|$;
	\item $\bar{q} = (q_\ell : \ell < k)$, $q_\ell \in \mathbb{Q}^+$;
	\item $\bar{x} = (x_\ell : \ell < k) \in \mrm{seq}_{k}(X_j)$;
	\item 
	\begin{enumerate}[$(\cdot_1)$]
	\item $\bar{y} = (y_\ell : \ell < k) \in \mrm{seq}_{k}(X)$
	\item $\bar{y} \subseteq X_{n} \cap \mrm{dom}(f_t) \setminus X_{< n}$;
	\item $\bar{y} = \bar{z}_{(n, t, \bar{x})}$ (recall \ref{hyp_A2_ho}(\ref{thezs}));
	\item $a = \sum \{q_\ell y_\ell : \ell < k\}$;
	\end{enumerate}
	\item\label{db} $q_\ell \in \{\frac{m}{n!}, -\frac{m}{n!} : m \in \{1, ..., (n!)^2\} \}$;
	\item $p^m \vert n!$;
	\item in $G_{(1, j)}$, $p^m$ divides $\sum \{q_\ell f_t(y_\ell) : \ell < k\}$, i.e., we have that:
	$$\frac{1}{p^m} \sum \{q_\ell f_t(y_\ell) : \ell < k\} \in G_{(1, j)};$$
	\item $p \leq n$.
	\end{enumerate}
	\end{enumerate}	
	\end{enumerate}
\end{definition}


	\begin{lemma}\label{lemma_pre_main_th_ho}\begin{enumerate}[(1)]
	\item\label{G1p_pure_co-hop} If $p = p_a$, $a \in G^+_0$, then:
	$$G_{(1, p)} = \langle b \in G^+_0 : b \leq_* a \rangle^*_{G_1},$$
where we recall that $\{ b \in G^+_0 : b \leq_* a \}$ is a finite set (by Claim~\ref{observation_Xtree_ho}(1)).
	\item For $t \in T$, $H_{(1, t)} := H_{(2, t)} \cap G_1$ and $I_{(1, t)} := I_{(2, t)} \cap G_1$ are pure in $G_1$.
	\item\label{convention_hatf_G1_ho} For $\hat{f}_t$ as in Definition~\ref{def_G02_hopfian}(\ref{hatf_on_G2_ho}), we have that $\hat{f}_t [H_{(1, t)}] \subseteq I_{(1, t)}$, consequently when we talk about $\hat{f}_t$ as a map on $G_1$ we actually mean $\hat{f}_t \restriction H_{(1, t)}$.
	\item\label{item4} $\hat{f}_t [H_{(1, t)}] \neq I_{(1, t)}$.
	\item If $a \in G^+_0$ and $p$ is a prime, then:
	$$G_1 \models p^\infty \vert \, a \; \Leftrightarrow \; G_{(1, 0)} \models p^\infty \vert \, a.$$
\end{enumerate}
\end{lemma}

	\begin{proof} (1) is proved similarly to \cite[5.17(1)]{1205}. The rest is easy, unraveling \ref{def_G1_hopfian}.
\end{proof}

	\begin{theorem}\label{main_th_hopfian} Let $\mathfrak{m}(T) \in \mathrm{K}^{\mrm{ho}}_1(T)$.
	\begin{enumerate}[(1)]
	\item We can modify the construction so that $G_1[\mathfrak{m}(T)] = G_1[T]$ has domain $\omega$ and the function $T \mapsto G_1[T]$ is Borel (for $T$ a tree with domain $\omega$).
	\item $T$ has an infinite branch iff $G_1[T]$ is not Hopfian.
\end{enumerate}
\end{theorem}

	\begin{proof} Items (1) is easy. Concerning the ``left-to-right" direction of items (2), let $\bar{t} = (t_n : n < \omega)$ be an infinite branch of $T$. Then $\hat{f} = \bigcup_{n < \omega} \hat{f}_{t_n}$ is an endomorphism of $G_2$ into $G_2$, since $G_2 = \bigcup_{n < \omega} H_{(2, t_n)}$, where $(H_{(2, t_n)} : n < \omega)$ is chain of pure subgroups of $G_2$ with limit $G_2$ (cf. Definition~\ref{def_G02_hopfian}(\ref{hatf_on_G2_ho})). Moreover, $\hat{f}$ is onto $G_2$, by Definition~\ref{def_G02_hopfian}(\ref{hatf_on_G2_ho}). Now we claim:
	\begin{enumerate}[$(*_{1})$]
	\item
	\begin{enumerate}[(a)]
	\item $\hat{f} \restriction G_1 = \bigcup_{n < \omega} \hat{f}_{t_n} \restriction H_{(1, t_n)}$ is an endomorphism of $G_1$;
	\item $\hat{f} \restriction G_1$ is onto;
	\item $\hat{f} \restriction G_1$ is not one-to-one.
	\end{enumerate}
	\end{enumerate}
We prove item $(*_{1})(a)$. We need to show that $\hat{f}$ maps $G_1$ into $G_1$. 
Clearly $\hat{f}$ maps $X$ into $X \cup \{0\}$, hence it maps $G_0$ into $G_0$. Thus, we have to prove:
	\begin{enumerate}[$(*_{1.1})$]
	\item If $a \in G^+_0$, $n < \omega$ and $p^m \in \mathbb{P}_{(a, n)}$, then $\frac{\hat{f}(a)}{p^m} \in G_1$.
	\end{enumerate}
We prove this on induction on $n < \omega$. It suffices to consider the following two cases.
\newline \underline{Case 1}. $n = 0$ and $p^m \in \mathbb{P}_{(a, 0)}$.
\newline By \ref{def_G1_hopfian}(1)(c)(ii), $p = p_b$ for some $b$ such that $a \leq_* b$ and so, for some $k < \omega$ large enough we have that $\hat{f}(a) = \hat{f}_{t_k}(a) \leq_* a \leq_* b$ and $\leq_*$ is a partial order, again by \ref{def_G1_hopfian}(1)(c)(ii), we are done.
\newline \underline{Case 2}. $n = j+1$ and $\frac{a}{p^m} \in \mathbb{P}_{(a, n)} \setminus \mathbb{P}_{(a, j)}$.
\newline Let then $(n, t, (q_\ell, x_\ell, y_\ell : \ell < k))$ be as in \ref{def_G1_hopfian}(1)(d) and recall that $\bar{t} = (t_n : n < \omega)$ was fixed at the beginning of the proof. Also, necessarily $t_{n+1} \notin T_n$.
\newline \underline{Case 2A}. $t \leq_T t_{n}$. By \ref{def_G1_hopfian}(1)(d), recalling $a = \sum\{q_\ell y_\ell : \ell < k\}$, we have that:
$$\hat{f}_{t_{n}} (\frac{1}{p^m}a) = p^{-m} \sum q_\ell f_{t_{n}}(y_\ell) = p^{-m} \sum q_\ell f_{t}(y_\ell) \in G_{(1, j)}.$$
\underline{Case 2B}. $t \not\leq_T t_{n}$. As $\bar{y} \subseteq \mrm{dom}(f_t) \cap X_{n} \setminus  X_{< n}$, by clause (v)$(\cdot_3)$ of \ref{def_G1_hopfian}(1)(d) we have $(y_\ell : \ell < k) = \bar{z}_{(n, t, \bar{x})}$, where $\bar{x} = (x_\ell : \ell < k)$. Hence, by \ref{hyp_A2_ho}(\ref{thezs}), we have $\bar{z}_{(n, t, \bar{x})} \in \mrm{seq}_{\mrm{lg}(\bar{x})}(\mrm{dom}(f_t) \setminus X_{< n})$. Thus, $\bar{z}_{(n, t, \bar{x})} \in \mrm{dom}(f_t) \setminus X_{< n}$, but we have:
	$$\bigcup_{i < \omega} f_{t_i} \restriction X_n = f_{t_n} \restriction X_n.$$
	Let $r = t \wedge t_n$, then $\mrm{dom}(f_t) \cap \mrm{dom}(f_{t_n}) = \mrm{dom}(f_r)$. As $t \not\leq_T t_n$, necessarily $r \neq t$ and so $\mathbf{n}(r) < n$, so $\bar{z}_{(n, t, \bar{x})} \subseteq \mrm{dom}(f_t) \setminus \mrm{dom}(f_r)$, hence $\bar{z}_{(n, t, \bar{x})}$ is disjoint from $\mrm{dom}(f_{t_n})$, but $X_n \subseteq \mrm{dom}(f_{t_n})$, hence $\hat{f}_{t_n}(\bar{z}_{(n, t, \bar{x})}) = (0, ..., 0)$, and so $\hat{f}(\bar{z}_{(n, t, \bar{x})}) = (0, ..., 0)$. So we have the following:
$$\hat{f}(a) = \hat{f}(\sum_{\ell < k} q_\ell y_\ell) = \sum_{\ell < k} q_\ell \hat{f}_{t_{n+1}}(\bar{z}_{(n, t, \bar{x})}) = 0.$$

\smallskip 
\noindent This proves $(*_{1})$(a), we thus move to $(*_{1})$(b). To this extent, by algebra it suffices to prove that if $b \in G^+_1$, $p^{-m}b \in G_1$, then for some $c \in G_1$ we have that for infinitely many $k < \omega$, $\hat{f}_{t_k}(c) = b$. To this extent, let $b = \sum_{\ell < k} q_\ell x_\ell$ with $(x_\ell : \ell < k) \in \mrm{seq}_k(X)$ and $q_\ell \in \mathbb{Q}^+$, for $\ell < k$. Let then $n_0$ be large enough so that we have: 
\begin{enumerate}[$(*_{1.2})$]
	\item 
	\begin{enumerate}[$(\cdot_1)$]
	\item if $\ell < k$, then $q_\ell \in \{\frac{i}{n_0!}, -\frac{i}{n_0!} : i \in \{1, ..., (n_0!)^2\}\}$ and $p^{m} \vert \; n_0$;
	\item $\{x_\ell : \ell < k\} \subseteq X_{n_0}$.
	\end{enumerate}
	\end{enumerate}
Let now $n_1 > n_0!$ be such that $\mathbf{n}(t_{n_1}) > n_0$. As $\mathbf{n}(t_{n_1+1}) > \mathbf{n}(t_{n_0})$, for some $n < \omega$ we have $t_{n_1+1} \in T_{n+1} \setminus T_n$. Now using \ref{def_G1_hopfian}(1)(d) for $t = t_{n_1+1}$ we are done.

\smallskip 
\noindent Finally, concerning clause $(*_1)$(c) it suffices to observe that for every $n < \omega$ such that $t_n \in T_{n+1} \setminus T_n$ we have that $X_{n+1} \supsetneq X_n \cup \bigcup \{\mrm{dom}(f_t) : t \in T_{n+1} \setminus T_n\}$ by condition \ref{hyp_A2_ho}(\ref{hyp_A2_ho_mapped_to_0}) and $\hat{f}_{t_n}$ maps any member in this set to $0$ by \ref{def_G02_hopfian}(\ref{hatf_on_G2_ho}).

\medskip 
\noindent We now prove the ``right-to-left" direction of item (2). To this extent, suppose:
	\begin{enumerate}[$(*_2)$]
	\item $T$ is well-founded and let $G_1[T] = G_1$ and $\pi \in \mrm{End}(G_1)$ be onto.
\end{enumerate}	 
We shall prove that $\pi$ is $1$-to-$1$. Toward contradiction, suppose not. Now:
\begin{enumerate}[$(*_3)$]
	\item If $x \in X$, then $\mrm{supp}(\pi(x)) \subseteq X_{\leq_* x} = \{y \in X : y \leq_* x\}$.
\end{enumerate}	
[Why? Let $p = p_x$, so $x \in G_{(1, p)}$, thus $\pi(x) \in G_{(1, p)}$, then apply \ref{lemma_pre_main_th_ho}(\ref{G1p_pure_co-hop}).]
	\begin{enumerate}[$(*_4)$]
	\item If $x \neq y \in X$ are $\leq_*$-incomparable and $x \in \mrm{supp}(\pi(x))$, then $y \in \mrm{supp}(\pi(y))$. 
\end{enumerate}
We prove $(*_4)$. Let $p = p_{x+y}$, so $x+y \in G^+_0$, and even $x+y \in G_{(1, p)}$, hence by \ref{lemma_pre_main_th_ho}(\ref{G1p_pure_co-hop}), $\pi(x+y)$ has the form $\sum \{r_i f_{\bar{t}_i}(x+y) : i < i(*) \}$, where $r_i \in \mathbb{Q}^+$ and $\bar{t}_i \in T^{n(i)}$, for some $n(i) \geq 0$ such that $(\bar{t}_i : i < i(*))$ is with no repetitions. So:
	\begin{enumerate}[$(*_{4.1})$]
	\item $\pi(x+y) = \pi(x) + \pi(y) = \sum_{i < i(*)} r_i f_{\bar{t}_i}(x) + \sum_{i < i(*)} r_i f_{\bar{t}_i}(y)$.
	\end{enumerate}
Now, by assumption, $x \in \mrm{supp}(\pi(x))$, whereas $x \notin \mrm{supp}(\pi(y))$, because by $(*_3)$, $\mrm{supp}(\pi(y)) \subseteq X_{\leq_* y}$ and $x \not\leq_* y$. Together $x$ belongs to the support of $\sum_{i < i(*)} r_i f_{\bar{t}_i}(x)$. Thus,  for some $j < i(*)$, $f_{\bar{t}_j}(x) = x$, and so $\bar{t}_j = ()$. As $(\bar{t}_i : i < i(*))$ is with no repetitions, we have that $i < i(*)$ and $j \neq i$ implies $\mrm{lg}(\bar{t}_j) > 0$, so $f_{\bar{t}_i}(y) = y$, but $f_{\bar{t}_j}(y) \neq y$, hence $y$ appears exactly once in the second term of the RHS of $(*_{4.1})$, but as before it does not appear in the first term of the LHS, hence $y$ appears in the support of the LHS, i.e., $\pi(x) + \pi(y)$ but $y \notin \mrm{supp}(\pi(x))$. So $y \in \mrm{supp}(\pi(y))$.
	\begin{enumerate}[$(*_5)$]
	\item For some $x \in X$, $x \notin \mrm{supp}(\pi(x))$.
	\end{enumerate}
Why $(*_5)$? Otherwise for every $a \in G^+_1$, let $a = \sum\{q_\ell x_\ell : \ell < k\}$, $k > 0$, $q_\ell \in \mathbb{Q}^+$ and $(x_\ell : \ell < k)$ with no repetitions. \begin{enumerate}[$(*_{5.1})$]
	\item W.l.o.g., $\ell < k-1$ implies $x_{k-1} \not\leq_* x_\ell$.
	\end{enumerate} 
	\begin{enumerate}[$(*_{5.2})$]
	\item $x_{k-1} \in \mrm{supp}(\pi(x_{k-1}))$.
	\end{enumerate} 
	[Why? By our assumption toward contradiction concerning $(*_5)$.]
	\begin{enumerate}[$(*_{5.3})$]
	\item If $\ell < k-1$, then $\mrm{supp}(\pi(x_{k-1})) \subseteq \{y \in X : y \leq_* x_\ell \}$.
	\end{enumerate} 
	[Why? By $(*_3)$.]
	\begin{enumerate}[$(*_{5.4})$]
	\item If $\ell < k-1$, then $x_{k-1} \notin \mrm{supp}(\pi(x_{\ell}))$.
	\end{enumerate} 
	[Why? By $(*_{5.1})$ and $(*_{5.3})$.]
	\begin{enumerate}[$(*_{5.5})$]
	\item $x_{k-1} \in \mrm{supp}(\pi(x_{k-1})) \setminus \bigcup_{\ell < k-1} \mrm{supp}(\pi(x_\ell))$.
	\end{enumerate} 
	[Why? By $(*_{5.1})$+$(*_{5.4})$.]
	\begin{enumerate}[$(*_{5.6})$]
	\item $x_{k-1} \in \mrm{supp}(\pi(a))$.
	\end{enumerate} 
	[Why? By $(*_{5.5})$ as $a = \sum_{\ell < k} q_\ell x_\ell$, $q_\ell \in \mathbb{Q}^+$.]
	\newline By $(*_{5.6})$, $\mrm{supp}(a) \neq \emptyset$, and so $\pi(a) \neq 0$. As $a$ was an arbitrary member of $G^+_1$, this implies that $\pi$ is $1$-to-$1$, contrary to our assumption.
\begin{enumerate}[$(*_6)$]
	\item For every $x \in X$, $x \notin \mrm{supp}(\pi(x))$.
	\end{enumerate}
Why $(*_6)$? Suppose not. Fixing $x \in X$, we can find $y \in X$ such that $y \notin \mrm{supp}(f(y))$ (otherwise $(*_6)$ holds). By \ref{observation_Xtree_ho}(\ref{observation_Xtree_ho7}), let $z \in X$ be such that $z$ is $\leq_*$-incomparable with both $x$ and $y$. By $(*_4)$ applied to the pair $(y, z)$, we deduce that $z \notin \mrm{supp}(\pi(z))$ and by $(*_4)$ applied to the pair $(x, z)$, we deduce that $x \notin \mrm{supp}(f(x))$, as promised. 
\newline Notice now that by $(*_3)$+$(*_6)$ we have that:
\begin{enumerate}[$(*_7)$]
	\item For every $x \in X$, $\mrm{supp}(\pi(x)) \subseteq X_{<_* x} = \{y \in X : y <_* x\}$.
	\end{enumerate}
\begin{enumerate}[$(*_8)$]
	\item Let $Z_n = \{x \in X : n = |\{y \in X : y <_* x \}|$. For $x \in Z_n \neq \emptyset$, let $n_x = n$ and $y_{(x, n-1)} <_* \cdots <_* y_{(x, 0)} < x$ list the set $\{y \in X : y < x \}$.
\end{enumerate}
\begin{enumerate}[$(*_9)$]
	\item There are rationals $q_{(x, i)}$, for $(i < n_x)$, s.t. $\pi(x) = \sum \{q_{(x, i)} y_{(x, i)} : i < n_x\}$.
\end{enumerate}
Why $(*_9)$? For $x \in X_n$, by $(*_7)$, $(*_8)$, $\pi(x) \in \langle \{y \in X : y <_* x \}\rangle^*_{G_1}$, so $(*_9)$ holds, i.e., there are $(q_{(x, i)} : i < n) \in \mathbb{Q}^n$ such that:
$$\pi(x) = \sum \{q_{(x, i)}y_{(x, i)} : i < n_x\}.$$
\begin{enumerate}[$(*_{10})$]
	\item There are rationals $q_{(n, i)}$, $(i < n < \omega)$, such that $q_{(x, i)} = q_{(n, i)}$ when $n = n_x$.
\end{enumerate}
Why $(*_{10})$? Fix $n$, choose $x_1 \in Z_n$ and let $q_{(n, i)} = q_{(x_1, i)}$. Now, if $x_2 \in Z_n$ we can find $x_3 \in Z_n$ such that $\{y \in X : y \leq_* x_3\}$ is disjoint from $\{y \in X : y \leq x_1 \text{ or } y \leq_* x_2\}$. Now, because of the above, $(*_{10})$ follows from:
\begin{enumerate}[$(*_{11})$]
	\item If $x(1), x(2) \in Z_n$ and $\{y_{(x(1), i)} : i < n\} \cap \{y_{(x(2), i)} : i < n\} = \emptyset$, then for all $i < n$ we have that $q_{(x(1), i)} = q_{(x(2), i)}$.
\end{enumerate}
Why? By letting $a = x(1) + x(2)$ and $p = p_a$, hence we have:
$$\pi(a) = \sum \{q_{(x(1), i)} y_{(x(1), i)} + q_{(x(2), i)} y_{(x(1), i)}: i < n\},$$
 and $\pi(a) \in G_{(1, p)}$, and so we can continue as in earlier cases.
\begin{enumerate}[$(*_{12})$]
	\item $q_{(n, i)} = 0$, when $i < n < \omega$.
\end{enumerate}
[Why? We can find $x \in Z_n$ which is $<_*$-minimal, so $\pi(x) = 0$, by $(*_5)$.]
\begin{enumerate}[$(*_{13})$]
	\item For every $x \in X$, $\pi(x) = 0$.
\end{enumerate}
[Why? By $(*_{9})$, $(*_{10})$ and $(*_{11})$.]
\begin{enumerate}[$(*_{14})$]
	\item $\mrm{ran}(\pi) = \{0 \}$.
\end{enumerate}
So we get a contradiction, as $\pi$ was assumed to be onto.
\end{proof}

\begin{proof}[Proof of Theorem~\ref{main_th4_hopfian}] By \ref{main_th_hopfian}.
\end{proof}

\section{The co-Hopfian problem for mixed countable $2$-nilpotent groups}\label{sec_cohop}

	\begin{proviso}\label{the_proviso} A complete understanding of this section requires familiarity with the definitions of \cite[Section~5]{1205} and some of its proofs. We feel that this cannot be helped (without adding essentially the whole of \cite[Section~5]{1205} to the present paper).
\end{proviso}

	\begin{convention} In this section we (naturally) write $2$-nilpotent groups in multiplicative notations and abelian groups in additive notation. We invite the reader to adopt the appropriate point of view whenever relevant (e.g. in \ref{def_H2}(\ref{itemc})). 
\end{convention}

	\begin{convention}\label{the_mathfrakm_convention} \begin{enumerate}[(1)]
	\item $\mathfrak{m}$, $T$, $X^{\mathfrak{m}}$, $G_\ell[\mathfrak{m}]$ and $(p_a : a \in G^+_0)$ \mbox{are as in \cite[Sec.~5]{1205}.}
	\item When clear from the context, we may omit $\mathfrak{m}$ in $X^{\mathfrak{m}}$, $f^\mathfrak{m}_t$, $G_\ell[\mathfrak{m}]$.
	\item\label{def_k(n)} $\leq^2_X \subseteq X^2 \times X^2$ is the closure to a partial order on $X$ of the following set:
	$$\{((x_1, y_1), (x_2, y_2)) : x_1, y_1, x_2, y_2 \in X, \exists t \in T \text{ s.t. } (f_t(x_1), f_t(y_1)) = (x_2, y_2)\}.$$
	\item\label{def_E2)} $E_2 \subseteq X^2 \times X^2$ is the closure to an equivalence relation of the partial order $\leq^2_X$.
\item $\leq^1_X \subseteq X \times X$ and $E_1 \subseteq X \times X$ are defined similarly.	
\end{enumerate}
\end{convention}

	\begin{choice}\label{the_choice}
	\begin{enumerate}[(1)]
	\item $\mathbb{P}$ denotes the set of primes.
	\item\label{the_choice_primes} We choose $((x_p, y_p) : p \in \mathbb{P})$ such that:
	\begin{enumerate}[(a)]
	\item for every $p_1, p_2 \in \mathbb{P}$ we have $x_{p_1}, y_{p_2} \in X$;
	\item $(x_p : p \in \mathbb{P})^\frown (y_p : p \in \mathbb{P})$ are pairwise not $E_1$-equivalent;
	\item $x_p + y_p$ is not divisible by $p$ in $G_1 = G_1[\mathfrak{m}]$;
	\item if $p_1 \neq p_2 \in \mathbb{P}$, then $(x_{p_2}, y_{p_2}), (y_{p_2}, x_{p_2}) \notin (x_{p_1}, y_{p_1})/E_2 \cup (y_{p_1}, x_{p_1})/E_2$.
	\end{enumerate}
	\item For $x, y \in X$ and $n < \omega$, we define $K_{(x, y)} = K_{(y, x)}$  and $a_{(x, y, n)}$ such that:
	\begin{enumerate}[(a)]
	\item if $x = y$, then $K_{(x, y)}$ is the trivial group $\{0\}$ and and $a_{(x, y, n)} = 0$;
	\item if $x \neq y$, $(x, y)$ is not $E_2$-equivalent to any member of $\{(x_p, y_p), (y_p, x_p) : p \in \mathbb{P}\}$, then $K_{(x, y)}$ is the trivial group $\{0\}$ and $a_{(x, y, n)} = 0$;
	\item\label{Kxy} if $(x, y) \in (x_p, y_p)/E_2$ or $(x, y) \in (y_p, x_p)/E_2$, then $K_{(x, y)} = K_p$ is $\mathbb{Z}^\infty_p$, i.e., the divisible $p$-group of rank $1$, and $a_{(x, y, n)} \in K_p$, for $n \geq 1$, satisfies:
	\begin{enumerate}
	\item\label{pa=0} if $n = 1$, then $K_p \models pa_{(x, y, n)} = 0 \neq a_{(x, y, n)}$;
	\item if $n = m+1$, then $K_p \models p a_{(x, y, n)} = a_{(x, y, m)}$.
	\end{enumerate}
	\end{enumerate}
\end{enumerate}
\end{choice}

	\begin{remark} The fact that we can fulfill choice \ref{the_choice} is easy to verify reading \cite[Section~5]{1205}, on which we rely, as said in Proviso~\ref{the_proviso}.
\end{remark}

	\begin{definition}\label{def_H2} For $\mathfrak{m}$ as in Conv.~\ref{the_mathfrakm_convention} we define a group $H_2[\mathfrak{m}] = H_2$ as follows:
	\begin{enumerate}[(1)]
	\item $H_2$ is generated by:
	\begin{enumerate}[(a)]
	\item $(\frac{1}{n!}, x)$, for $x \in X$ and $0 < n< \omega$;
	\item $z_{(x, y, d)}$, for $x, y \in X$ and $d \in K_{(x, y)}$;
	\end{enumerate}
	\item\label{def_H2_1} freely except for the following equations:
	\begin{enumerate}
	\item $z_{(x, y, 0)} = e$ and $z_{(x, x, d)} = e$;
	\item\label{z's_commute} $[z_{(x_1, y_1, d_1)}, z_{(x_2, y_2, d_2)}] = e$;
	\item\label{itemc} $z^{-1}_{(x, y, d_1)} = z_{(y, x, d_2)}$, when $d_1 = - d_2 \in K_{(x, y)}$;
	\item\label{di's_item} $z_{(x, y, d_1)} z_{(x, y, d_2)} = z_{(x, y, d)}$ , when $d_1, d_2 \in K_{(x, y)}$ and $K_{(x, y)} \models d = d_1 + d_2$;
	\item\label{powers_item} $(\frac{1}{(n+1)!}, x)^{n+1} = (\frac{1}{n!}, x)$;
	\item\label{using_k(n)} $[(\frac{1}{n!}, x), (\frac{1}{n!}, y)] = z_{(x, y, a_{(x, y, n)})}$ (cf. Choice~\ref{the_choice}(\ref{def_k(n)}));
	\item\label{commuting_item} $z_{(x, y, d)}t = t z_{(x, y, d)}$, for $x, y \in X$, $d \in K_{(x, y)}$, and $t$ belonging to:
	$$\{(\frac{1}{n!}, x) : x \in X, 0 < n< \omega \}.$$
	\end{enumerate}
	\item\label{def_H2Y} For $Y \subseteq X$, we let $H_2[Y]$ be the subgroup of $H_2$ generated by the set:
	$$\{(\frac{1}{n!}, x) : n < \omega, \; x \in Y\}.$$
	\item\label{g2t} For $t \in T$, we let $g^2_t$ be the partial automorphism of $H_2$ satisfying:
	$$n < \omega \text{ and } x \in X \Rightarrow g^2_t((\frac{1}{n!}, x)) = (\frac{1}{n!}, f_t(x)),$$
where $f_t$ is the partial map from $X$ into $X$ from \cite[Section~5]{1205}.
\end{enumerate}	 
\end{definition}

	\begin{remark} Notice that by the choice we made in \ref{the_choice}(\ref{Kxy}) and the equations we imposed in the definition of $H_2$ from \ref{def_H2}(\ref{def_H2_1}) we have that $H_2$ is generated by $\{(\frac{1}{n!}, x) : n < \omega, \; x \in X \}$ (to see this notice that $\{ a_{(x, y, n)} : n < \omega \}$ generates $K_{(x, y)}$); this is also remarked in \ref{g17}(\ref{g17_1}). This makes the definition of $H_2[Y]$ from \ref{def_H2}(\ref{def_H2Y}) and the definition of $H_2$ consistent, i.e., we naturally have that $H_2 = H_2[X]$.
\end{remark}

\begin{cclaim}\label{4.3A}
	\begin{enumerate}[(1)]
	\item 
	\begin{enumerate}
	\item $H_2[\mathfrak{m}]$ is nilpotent of class $2$ and its center is generated by: $$Z = \{z_{(x, y, d)} : x, y \in X \text{ and } d \in K_{(x, y)} \};$$ 
	\item $K'_{(x, y)} = \{z_{(x, y, d)} : d \in K_{(x, y)}\}$ is a subgroup of $H_2$ canonically isomorphic to the $K_{(x, y)}$ from \ref{the_choice}, hence from now on we might identify them;
	\end{enumerate} 
	\item 
	\begin{enumerate}
	\item In $\langle Z \rangle_{H_2}$ we have $z_{(x_1, y_1, d_1)} = z_{(x_2, y_2, d_2)}$ iff $(x_1, y_1, d_1) = (x_2, y_2, d_2)$;
	\item $\langle Z \rangle_{H_2} = \bigoplus \{K_{(x, y)} : x, y \in X \} \in \mrm{AB}$ and it is divisible group.
	\end{enumerate}
	\item Fixing any linear order $<_1$ of $X$ and any linear order $<_2$ of $X \times X$, every member of $H_2$ can be written as a product of the following form (so $k, m < \omega$):
	$$(\frac{1}{n!}, x_0)^{q_0} \cdots (\frac{1}{n!}, x_{k-1})^{q_{k-1}} z_{(x_0, y_0, d_0)} \cdots z_{(x_{m-1}, y_{m-1}, d_{m-1})}$$
in such a way that:
\begin{enumerate}[(a)]
	\item $x_0 <_1 \cdots <_1 x_{k-1}$;
	\item $((x_i, y_i) : i < m)$ is $<_2$-increasing and, for every $i < m$, $K_{(x_i, y_i)} \models d_i \neq 0$;
	\item $q_\ell \in \mathbb{Z}^+$;
	\item if $n > 1$, then for some $\ell < k$, $\frac{q_\ell}{n} \notin \mathbb{Z}$;
	\item for every $i < m$, $x_i <_1 y_i$.
\end{enumerate}
	\item The representation of elements of $H_2$ from (3) is unique. Notice that this representation  depends on $<_1$ and $<_2$ and it is unique only after fixing them.
	\item $H_2[Y]$ is a group generated by $\{ (\frac{1}{n!}, x) : x \in Y, 0 < n< \omega\} \cup \{ z_{(x, y, d)} : x, y \in Y \text{ and } d \in K_{(x, y)} \}$ freely except for the equations in \ref{def_H2}(\ref{def_H2_1}) relativized to $Y$.
	\item For $t \in T$, $g^2_t$ (cf. \ref{def_H2}(\ref{g2t})) is a partial embedding of $H_2[\mrm{dom}(f_t)]$ into $H_2[\mrm{ran}(f_t)]$.
\end{enumerate}
\end{cclaim}

\begin{proof} Easy, e.g. concerning (3), let $a \in H_1^+$ and let it be written as $w_0 \cdots w_{i(*)-1}$ with $i(*) \geq 1$ and, for $i < i(*)$, $w_i$ an integer power $\neq 0$ of some generator. By \ref{def_H2}(\ref{powers_item})(\ref{commuting_item}), w.l.o.g. we can assume that there is $j(*) < i(*)$ and  $0 < n < \omega$ s.t.:
\begin{equation}\tag{$\star$} a = (\frac{1}{n!}, x_1)^{q_1} \cdots (\frac{1}{n!}, x_{j(*)-1})^{q_{j(*)-1}} z_{(x_{j(*)}, y_{j(*)}, d_{j(*)})} \cdots z_{(x_{i(*)-1}, y_{i(*)-1}, d_{i(*)-1})}.
\end{equation}
Among the representations of $a$ as in $(\star)$ we choose one minimizing $k = |\{(i, j) : i < j < j_* \text{ but } x_j \leq_1 x_i\}|$. If $(x_i : i < j_*)$ is $<_1$-increasing, then we are essentially done, as the $z_{(x, y, d)}$'s are pairwise commuting, by \ref{def_H2}(\ref{z's_commute}), and $z_{(x, x, d)} = e$ and $z_{(x, y, d)} = z_{(y, x, -d)}$.
If $(x_i : i < j_*)$ is not $<_1$-increasing, then, as $\leq_1$ is a linear order of $X$, for some $i < j_*$ we have that $x_{i+1} \leq_1 x_i$, and necessarily $x_{i+1} <_1 x_i$, as if $x_{i+1} = x_i$ we can replace $(\frac{1}{n!}, x_i)^{q_i}(\frac{1}{n!}, x_{i+1})^{q_i+1}$ by $(\frac{1}{n!}, x_i)^{q_i+q_{i+1}}$ and $q_i + q_{i+1} \in \mathbb{Z}$ and so $k = |\{(i, j) : i < j < j_* \text{ but } x_j \leq_1 x_i\}|$ decreases, a contradiction. Now, making a case distinction on whether $q_i, q_{i+1}$ are positive or negative, we can use \ref{def_H2}(\ref{using_k(n)}) to prove by induction on $q_i + q_{i+1}$, that we can change the order of $(\frac{1}{n!}, x_i)^{q_i}$ and $(\frac{1}{n!}, x_{i+1})^{q_{i+1}}$ at the cost of an element of the form $z_{(x, y, d)}$, which goes to the right of the \mbox{expression $(\star)$ (recall that the $z$'s are in the center of $H_2$, see item~(1)).}

\smallskip
\noindent Concerning item (6), let $Y_1 = \mrm{dom}(f_t)$ and $Y_2 = \mrm{ran}(f_t)$. Then, firstly, $f_t$ maps $Y_1$ into $Y_2$ and it is $1$-to-$1$. To conclude it suffices to observe the following:
	$$(x, y) \in (x_p, y_p)/E_2 \; \Leftrightarrow \; (f_t(x), f_t(y)) \in (x_p, y_p)/E_2,$$
where we invite the reader to recall the definition of $E_2$ from  \ref{the_mathfrakm_convention}(\ref{def_E2)}).
\end{proof}


\begin{cclaim}\label{quotient_claim}  There is an onto homomorphism $h_2 : H_2 \rightarrow G_2$ such that:
	$$(\frac{1}{n!}, x) \mapsto \frac{1}{n!}x \in G_2 \text { and } z_{(x, y, d)} \mapsto e.$$
\end{cclaim}

\begin{definition}\label{def_H10}
	\begin{enumerate}[(1)]
	\item $H_0[\mathfrak{m}] = H_0$ is the subgroup of $H_2$ generated by:
	$$\{(1, x) : x \in X \}.$$
	\item\label{def_H1} $H_1[\mathfrak{m}] = H_1[T] = H_1 = \{a \in H_2[\mathfrak{m}] : h_2(a) \in G_1[\mathfrak{m}]\}$. 
	\item For $Y \subseteq X$ and $\ell \in \{0, 1\}$, we let $H_\ell[Y] = H_\ell \cap H_{2}[Y]$.
	\item For $t \in T$ and $\ell \in \{0, 1, 2\}$, let $g^\ell_t = g^2_t \restriction H_\ell$.
\end{enumerate}	 
\end{definition}

\begin{cclaim}\label{quotient_claim_post} 
\begin{enumerate}[(1)]
	\item For $\ell \in \{0, 1\}$, we define $h_\ell = h_2 \restriction H_\ell$ (see \ref{def_H10}).
	\item For $\ell \in \{1, 2\}$, the homomorphism $h_\ell : H_\ell \rightarrow G_\ell$ is onto.
	\item\label{cong_gl} For $\ell \in \{1, 2\}$, $\mrm{ker}(h_\ell) = \mrm{Cent}(H_\ell) = \langle Z \rangle_{H_2} \cap H_\ell$, so $H_\ell/\mrm{Cent}(H_\ell) \cong G_\ell$.
	\item\label{quotient4} For $\ell \in \{1, 2\}$, the set $\{ (1, x)/\mrm{Cent}(H_\ell) : x \in X\}$ form a basis for  $H_\ell/\mrm{Cent}(H_\ell)$.
\end{enumerate}	
\end{cclaim}

	\begin{cclaim}\label{g17}
	\begin{enumerate}[(1)]
	\item\label{g17_1} $H_2$ is generated by $\{(\frac{1}{n!}, x) : 0 < n < \omega, x \in X\}$.
	\item $g^2_t \restriction H_1 =: g^1_t =: g_t$ (cf. \ref{def_H2}(\ref{g2t})) is a partial automorphism of $H_1$.
	\item If $(t_n : n < \omega)$ is an infinite branch of $T$, then $\bigcup_{n < \omega} g^1_{t_n}$ is an embedding of $H_1$ into itself which is not onto, so $H_1$ is not co-Hopfian.
	\item The commutator subgroup of $H_1$ is its center (so $H_1 \in \mrm{NiGp}(2)$) and this subgroup is the intersection of $H_1$ with the center of $H_2$, i.e., $\langle Z \rangle_{H_2}$.
	\end{enumerate}
\end{cclaim}

	\begin{proof} Items (1)-(3) and (5) are easy. We prove (4). For ease of notation, for the remainder of the proof, for $\ell \in \{1, 2\}$, we let $g^\ell_{t_n} = g^\ell_n$ and $\bigcup_{n < \omega} g^\ell_{t_n} = g^\ell$. Similarly, we let $f^\ell_{t_n} = f^\ell_n$ and $\bigcup_{n < \omega} f^\ell_{t_n} = f^\ell$. So for $\ell \in \{1, 2\}$ and $n < \omega$ we have:
	\begin{enumerate}[$(*_0)$]
	\item $h_\ell \circ g^\ell_n = h_\ell \circ f^\ell_n$.
	\end{enumerate}
	 Now, we have the following:
	\begin{enumerate}[$(*_1)$]
	\item Each $g^\ell_n$ is a partial automorphism of $H_\ell$.
	\end{enumerate}
	[Why? By Item (2) of the present claim.]
	\begin{enumerate}[$(*_2)$]
	\item If $n < m$, then $g^\ell_n \subseteq g^\ell_m$.
	\end{enumerate}
	\begin{enumerate}[$(*_3)$]
	\item $g^\ell$ is a partial automorphism of $H_\ell$.
	\end{enumerate}
	\begin{enumerate}[$(*_4)$]
	\item For every $n < \omega$, $g^1_n \subseteq g^2_n$.
	\end{enumerate}
	\begin{enumerate}[$(*_5)$]
	\item $g^1 \subseteq g^2$.
	\end{enumerate}
	\begin{enumerate}[$(*_6)$]
	\item $\mrm{dom}(g^2) = H_2$.
	\end{enumerate}
	[By \cite[Section~5]{1205} and the definition of $g^2_t$ from \ref{def_H2}(\ref{g2t}) of the present claim, each $(\frac{1}{n!}, x) \in \mrm{dom}(g^2)$ and by Item (1) of the present claim such elements generate $H_2$.]
	\begin{enumerate}[$(*_7)$]
	\item $\mrm{dom}(g^1) = H_1$.
	\end{enumerate}
	[Why? By $(*_6)$ and the fact that for each $n < \omega$ we have that $g^1_n = g^2_n \restriction H_1$.]
	\begin{enumerate}[$(*_8)$]
	\item There is $x_* \in X$ such that $x \notin \mrm{ran}(f^1)$.
	\end{enumerate}
	[Why? Let $x_* \in X_0$, by \cite[5.4(d)]{1205}, for all $t \in T$, $x_* \notin \mrm{ran}(f^2_t)$ and so $x_* \notin \mrm{ran}(f^1_t)$.]
	\begin{enumerate}[$(*_9)$]
	\item $(1, x_*) \in H_1$ but $(1, x_*) \notin \mrm{ran}(f^1)$.
	\end{enumerate}
	[Why? $(1, x_*) \in H_1$ by $(*_7)$ and $(1, x_*) \notin \mrm{ran}(f^1)$ by $(*_8)$ and the definitions.]
	\begin{enumerate}[$(*_{10})$]
	\item $g^1$ is an embedding of $H_1$ into itself which is not onto.
	\end{enumerate}
\end{proof}

	\begin{cclaim}\label{2_nilpotent_cohop_claim} If the tree $T$ is well-founded, then the group $H_1[T] = H_1$ is co-Hopfian.
\end{cclaim}

	\begin{proof} Suppose that the tree $T$ is well-founded. We use Claim~\ref{g17} freely. Let $\pi$ be an embedding of $H_1$ into itself. As $\pi$ maps the set of commutators of $H_1$ into itself and this set generates the center of $H_1$ we have that $\pi$ maps the center of $H_1$ into itself. Furthermore, as $\pi$ is one-to-one, it maps non-central elements of $H_1$ to non-central elements of $H_1$. Hence, $\pi$ induces an embedding $\hat{\pi}$ of $H_1/\mrm{Cent}(H_1)$ into itself. Up to renaming, $\hat{\pi}$ is then an embedding of $G_1 = G_1[\mathfrak{m}]$ into itself (cf. \ref{quotient_claim_post}(\ref{cong_gl})).
\noindent By \cite[Section~5]{1205} there is $m \in \mathbb{Z}^+$ such that $\hat{\pi}$ satisfies the following:
$$(1, x)/\mrm{Cent(H_1)} \mapsto (1, x)^{m}/\mrm{Cent(H_1)} \text{ (cf. \ref{quotient_claim_post}(\ref{quotient4}))}.$$
We want to show that $m = 1$ or $m = -1$ (notice that in \cite[Section~5]{1205} we do not necessarily have this, but only that $m \in \mathbb{Z}$). Toward contradiction, suppose first that $m > 1$ and let $p$ be a prime dividing $m$. Let $x, y \in X$ be $(x_p, y_p)$, for the fixed prime $p$ from the previous sentence (recall \ref{the_choice}(\ref{the_choice_primes})). Then we have for $a = a_{(x, y, 1)}$:
\begin{equation} [(1, x), (1, y)] = z_{(x, y, a)} \neq e_{H_1} \text{ and } z_{(x, y, a)}^p = e \text{ (so } z_{(x, y, a)}^{m^2} = e) \; \text{(cf. \ref{the_choice}(\ref{pa=0}))}.
\end{equation}
Now, as $\hat{\pi} : (1, x)/\mrm{Cent(H_1)} \mapsto (1, x)^{m}/\mrm{Cent(H_1)}$, for every $x \in X$, there is $c_x \in \mrm{Cent(H_1)}$ such that $f((1, x)) = (1, x)^{m}c_x$. Hence, as $\pi$ is $1$-to-$1$, we have:
$$\begin{array}{rcl}
e_{H_1} & \neq & \pi([(1, x), (1, y)]) \\
		& = & [\pi((1, x)), f((1, y))] \\
		& = & [(1, x)^m c_x, (1, y)^m c_y] \\
		& = & [(1, x)^m, (1, y)^m] \\
	    & = & [(1, x), (1, y)]^{m^2} \text{ \; [by Fact~\ref{2nilpotent_fact}]}\\
	    & = & z_{(x, y, a)}^{m^2},
\end{array}$$
where $a = a_{(x, y, 1)}$, and this contradicts $z_{(x, y, a)}^{m^2} = e_{H_1}$ (cf. Equation (1)). So $m > 1$ is impossible, and similarly $m < -1$ is also not possible. Hence $m \in \{1, -1\}$.
\newline We now show that $\pi$ is onto, thus concluding that $H_1$ is co-Hopfian. First note:
	\begin{enumerate}[$(*_1)$]
	\item if $a, b \in H_1$, then for some $a_1, b_1 \in \mrm{ran}(f)$ we have $[a, b] = [a_1, b_1]$.
	\end{enumerate}
We prove $(*_1)$. As $\hat{\pi}$ is onto, we can find $a_1 \in \mrm{ran}(\pi)$ and $c \in \mrm{Cent}(H_1)$ such that $a = a_1c$. Similarly, we can find $b_1 \in \mrm{ran}(\pi) $ and $d \in \mrm{Cent}(H_1)$ such that $b = b_1d$. Thus:
$$\begin{array}{rcl}
[a, b]  & = & [a_1c, b_1d] \\
		& = & (a_1c)^{-1} (b_1d)^{-1} (a_1c) (b_1d)  \\
		& = & (a_1^{-1}b_1^{-1}a_1b_1) (c^{-1}d^{-1}cd)\\
		& = & (a_1^{-1}b_1^{-1}a_1b_1) e_{H_1} \\
	    & = & [a_1, b_1].
\end{array}$$
Hence, we have:
\begin{enumerate}[$(*_2)$]
	\item the set of commutators of $H_1$ is included in $\mrm{ran}(\pi)$.
	\end{enumerate}
As the set of commutators of $H_1$ generates $\mrm{Cent}(H_1)$ and $\hat{\pi}$ is onto, $\pi$ is onto.
\end{proof}

	\begin{proof}[Proof of Theorem~\ref{main_th7}] By Claims \ref{g17}(4)(5) and \ref{2_nilpotent_cohop_claim}.
\end{proof}

\section{The co-Hopfian problem for countable AB}\label{sec_general_aim}

	We need the following three claims from \cite{1214}.
	
	\begin{observation}[{\cite[Obs.~2.17]{1214}}]\label{suff_cond} Let $G \in \mrm{AB}$. Then $G$ is non-co-Hopfian iff:
	\begin{enumerate}[$(\star)$]
	\item there are $f$ and $z  \in G$ such that:
	\begin{enumerate}[(a)]
	\item $f \in \mrm{End}(G)$;
	\item $f(x) \neq x$ for every $x \in G \setminus \{ 0 \}$;
	\item for every $x \in G$, $z \neq x - f(x)$.
	\end{enumerate}
	\end{enumerate}
\end{observation}
	
	\begin{cclaim}[{\cite[Claim~3.2]{1214}}]\label{the_crucial_claim} Let $G \in \mrm{AB}$ be countable and reduced. Let also $p \in \mathbb{P}$, and suppose that $\mrm{Tor}_p(G)$ is infinite. Then:
	\begin{enumerate}[(1)]
	\item $G$ is not co-Hopfian;
	\item If in addition $\mrm{Tor}_p(G)$ is not bounded, then we can find $\bar{K}$ and $K$ such that:
	\begin{enumerate}
	\item $\bar{K} = (K_n : n < \omega)$ and $K = \bigoplus_{n < \omega} K_n \leq_* G$;
	\item $K_n \leq G$ is a non-trivial finite $p$-group;
	\item there is $f \in \mrm{End}(G)$ such that $\mrm{ran}(f) \subseteq K$ and for every $n < \omega$ we have that $\{0\} \neq f(K_n) \subseteq K_n$;
	\item $f$ is as in \ref{suff_cond}.
	\end{enumerate}
	\item If in addition to (2) $\mrm{Tor}_p(G)$ has height $\geq \omega$, then in (2)(b) we have that for some increasing $k(n)$,  $p^n(p^{k(n)}K_n) \neq \{ 0 \}$ and $x \in K_n \Rightarrow f(x) = p^{k(n)}(x)$.
	\end{enumerate}
\end{cclaim}

\begin{cclaim}[{\cite[Claim~3.6]{1214}}]\label{quotient_withT_divisible} Let $G \in \mathrm{AB}$ and $p \in \mathbb{P}$. If $\mrm{Tor}_p(G)$ is bounded and $G/\mrm{Tor}_p(G)$ is not $p$-divisible, then $G$ is not co-Hopfian.
\end{cclaim}

	\begin{definition}\label{dense_def} We say that a subgroup $H$ of a reduced abelian group $G$ is dense in the $\mathbb{Z}$-adic topology when $G/H$ is a divisible group. 
\end{definition}

	\begin{remark} Def.~\ref{dense_def} could be formulated using indeed topological (or even metric) notions, but by \cite[Sec.~6, Ex.~10(b)]{fuch_vol1} this is equivalent to the above definition.
\end{remark}

	\begin{cclaim}\label{i2} Let $G \in \mrm{AB}_\omega$ and suppose that $G$ is co-Hopfian, then:
	\item $G$ can be represented as $G_1 \oplus G_2$, where
	\begin{enumerate}[(a)]
	\item $G_1$ is divisible of finite rank (and hence co-Hopfian);
	\item $G_2$ is reduced and co-Hopfian;
	\item for $\mrm{Tor}(G_2) = K = \bigoplus \{K_{(2, p)} : p \in \mathbb{P} \} \leq G_2$, we have:
	\begin{enumerate}[$(\cdot_1)$]
	\item $K_{(2, p)} = \mrm{Tor}_p(G_2) = K_p$ is finite;
	\item $K$ is dense in $G_2$ in the $\mathbb{Z}$-adic topology;
	\item any $\pi \in \mrm{End}(G_2)$ maps $K_{(2, p)}$ to $K_{(2, p)}$;
	\item any $1$-to-$1$ $\pi \in \mrm{End}(G_2)$ maps $K$ onto itself (as $K$ is co-Hopfian);
	\item to any $a \in G_2$ we can associate a unique sequence $\{c_{(a, p)} : p \in \mathbb{P} \}$ s.t.:
	\begin{enumerate}[(i)]
	\renewcommand\labelitemi{--}
	\item\label{the_ca} $c_{(a, p)} \in K_{(2, p)}$;
	\item  in $G_2$, $a - c_{(a, p)}$ is divisible by $p^n$, for every $n < \omega$;
	\item $a \mapsto c_{(a, p)}$ is a projection from $G_2$ onto $K_{(2, p)}$, furthermore, this map is the identity on $K_p$ and $0$ on $K_q$ for $q \neq p$;
	\item $a - \sum \{c_{(a, p)} : p \in \mathbb{P}, p < n \}$ is divisible by $(n-1)!$;
	\item $a \mapsto (c_{(a, p)} : p \in \mathbb{P})$ embeds $G_2$ into $\prod \{K_{(2, p)} : p \in \mathbb{P} \}$. 
	\end{enumerate}
	\item if $\pi \in \mrm{End}(G_2)$ and $a \in G_2$, then $\pi(c_{(a, p)}) = c_{(\pi(a), p)}$, for $p \in \mathbb{P}$.
	\end{enumerate}
	\end{enumerate}\end{cclaim}

	\begin{proof} Let $G_1 \leq G$ be a maximal divisible subgroup of $G$, it is known that there is reduced $G_2 \leq G$ such that $G = G_1 \oplus G_2$.  Now, the fact that $G_1$ and $G_2$ are co-Hopfian is by \cite[Lemma~3]{b&p}. The fact that $G_1$ is as in (a) is well-known. Concerning (c), $(\cdot_1)$ is by \ref{the_crucial_claim}, $(\cdot_2)$ is by \ref{quotient_withT_divisible}, $(\cdot_3)$ is clear, and $(\cdot_4)$ is again by \cite[Lemma~3]{b&p}. We prove $(\cdot_5)$. To this extent, fix $a \in G_2$. For every prime $p$ and $n < \omega$ there is $c = c_{(a, p^n)} \in K$ such that $G_2 \models p^n \mid \, (a-c)$ (as $G_2/K$ is divisible, recalling \ref{quotient_withT_divisible}). Also, clearly $c = \sum \{c_q : q \text{ prime } q < q_*\}$ for some $q_* < \omega$ and $c_q \in  K_{(2, q)}$. But $q \neq p$ implies $p^n K_{(2, q)} = K_{(2, q)}$, so $p^n \mid c_q$. Thus w.o.l.g. $p \neq q$ implies $c_q = 0$. Hence, there is $c = c_{(a, p^n)} \in K_{(2, p)}$ as above. Now, as $K_{(2, p)}$ is finite, some $c \in K_{(2, p)}$ occurs for infinitely many $n$, hence w.l.o.g. $c_{(a, p^n)} = c_{(a, p)}$. So $G_2 \models p^n \mid \, (a-c_{(a, p)})$, for every $n < \omega$, as if $a-c_{(a, p)}$ is divisible by $p^n$, then it is divisible by $p^m$ for all $1 \leq m \leq n$. 
Also clearly for $n$ large enough $p^n K_{(2, p)} = { 0 }$, hence $c_{(a, p)}$ is unique. We elaborate on the uniqueness of $c_{(a, p)}$. Let $c_1, c_2 \in K_{(2, p)} = K_p$ be such that $p^\infty \, \vert \, (a - c_\ell)$, for $\ell \in \{1, 2\}$ and let $n < \omega$ be such that $p^nK_p = \{0\}$. Clearly, $p^n \, \vert \, (a - c_1), (a - c_2)$ and so we have the following condition:
	$$p^n \, \vert \, ((a - c_1) - (a - c_2)) = c_2 - c_1 := c_*.$$
Hence, there is $d \in G$ such that $p^nd = c_*$ but $c_* \in K_p = \mrm{Tor}_p(G)$, hence $d \in K_p$ (as $K_p$ is pure in $G$ being $K_p$ finite), which implies that $p^nd = 0 = c_*$ (as $n < \omega$ was chosen so that $p^nK_p = \{0\}$). Hence, $c_1 = c_2$, as wanted, i.e., $c_{(a, p)}$ is unique.

\smallskip
\noindent 
Now, $(\cdot_5)(i)$-$(iii)$ are clear from what we observed above, we prove $(\cdot_5)(iv)(v)$. Concerning $(\cdot_5)(iv)$, note that for every $p \in \mathbb{P}$ and $m < \omega$, $a - c_{(a, p)}$ is divisible by $p^m$, so let $b_{(p, m)} \in G_2$ be such that $p^m b_{(p, m)} = a - c_{(a, p)}$. Let now $m < \omega$ be large enough such that for every prime $p < n$ we have the following:
\begin{equation}\tag{$\star_1$}
p^m K_{(2, p)} = \{0\},
\end{equation} and let $k = k_m = \prod_{p \in \mathbb{P}_{< n}} p^m$, where $\mathbb{P}_{< n} = \{p \in \mathbb{P} : p < n \}$. For $p \in \mathbb{P}_{< n}$ we have:
\begin{equation}\tag{$\star_2$}
k b_{(p, m)} = \frac{k}{p^m}(a - c_{(a, p)}).
\end{equation}
Let $\mathcal{U}$ be the ideal of the ring $\mathbb{Z}$ generated by $\{\frac{k}{p^m} : p \in \mathbb{P}_{< n} \}$, we claim:
\begin{enumerate}[$(\star_{2.1})$]
	\item  $1 \in \mathcal{U}$.
\end{enumerate}
Why $(\star_{2.1})$? Let $\mathcal{U} = \mathbb{Z}a$ (it is well-known that every ideal of $\mathbb{Z}$ has this form). As $\frac{k}{p^m} \in \mathcal{U}$, necessarily $a \neq 0$, so w.l.o.g. $a > 0$ and $a \vert \frac{k}{p^m}$. \noindent
Hence, $a$ has the form $\prod \{p^{i(p)} : p \in \mathbb{P}_{< n}\}$. Now, for each $p \in \mathbb{P}_{< n}$, $a \vert \frac{k}{p^m}$, hence, for each $p \in \mathbb{P}_{< n}$, ${i(p)} = 0$, and so $a = 1$, as wanted.

\noindent 
\smallskip
Hence, by $(\star_{2.1})$, there are $\ell_p$, for $p \in \mathbb{P}_{< n}$ such that $\sum \{\ell_p \frac{k}{p^m} : p \in \mathbb{P}_{< n}\} = 1$. 
	\begin{enumerate}[$(\star_3)$]
\item We have the following:
$$\begin{array}{rcl}
a - \sum_{p \in \mathbb{P}_{< n}} c_{(a, p)} & = & 
        1(a - \sum_{p \in \mathbb{P}_{< n}} c_{(a, p)}) \\
  & = & (\sum_{p \in \mathbb{P}_{< n}} \ell_p \frac{k}{p^m}) (a - \sum_{p \in \mathbb{P}_{< n}} c_{(a, p)}) \\
  & = & \sum_{p \in \mathbb{P}_{< n}} (\ell_p \frac{k}{p^m} (a - \sum_{p \in \mathbb{P}_{< n}} c_{(a, p)})) \\
  & = & \sum_{p \in \mathbb{P}_{< n}} (\ell_p (\frac{k}{p^m} a - \frac{k}{p^m} \sum_{p \in \mathbb{P}_{< n}} c_{(a, p)})) \\
  & = & \sum_{p \in \mathbb{P}_{< n}} (\ell_p (\frac{k}{p^m} a - \frac{k}{p^m} c_{(a, p)})) \\
  & = & \sum_{p \in \mathbb{P}_{< n}} (\ell_p (\frac{k}{p^m} (a - c_{(a, p)})) \\
  & = &  \sum_{p \in \mathbb{P}_{< n}} (\ell_p k b_{(p, m)})\\
  & = &  k_m (\sum_{p \in \mathbb{P}_{< n}} \ell_p b_{(p, m)}),\\
\end{array}$$
\end{enumerate}
where above in order to pass from the fifth to the sixth term we used $(\star_1)$ and the choice of $m$, and to pass from the seventh to the eighth term we used $(\star_2)$. So, $k = k_m \vert \, a - \sum_{p \in \mathbb{P}_{< n}} c_{(a, p)}$, and so  $(\cdot_5)(iv)$ \mbox{holds, as for $m$ large enough $(n-1)! \vert \, k_m$.}

\smallskip
\noindent
Concerning  $(\cdot_5)(v)$, we shall show that if $a \in G_2$ and $\bigwedge_{p \in \mathbb{P}} c_{(p, a)} = 0$, then $a = 0$, as this shows that $a \mapsto (c_{(a, p)} : p \in \mathbb{P})$ embeds $G_2$ into $\prod \{K_{(2, p)} : p \in \mathbb{P} \}$. Now:
	\begin{enumerate}[$(\star_4)$]
	\item if $a \in G_2$ and $\bigwedge_{p \in \mathbb{P}} c_{(a, p)} = 0$, then for every prime $q$ for some $b \in G_2$ we have that $qb = a$ and $\bigwedge_{p \in \mathbb{P}} c_{(b, p)} = 0$.
\end{enumerate}
But before proving $(\star_4)$ we show:
\begin{enumerate}[$(\star_5)$]
	\item $(\star_4)$ suffices to establish $(\cdot_5)(v)$.
\end{enumerate}
[Why? Assuming that $a \in G_2$ and $\bigwedge_{p \in \mathbb{P}} c_{(p, a)} = 0$, using $(\star_4)$ we can choose $(b_n : n < \omega)$ such that $b_0 = a$ and $p_n b_{n+1} = b_n$, where $p_n \in \mathbb{P}$ and for all $p  \in \mathbb{P}$ for infinitely many $n < \omega$ we have that $p_n = p$. Hence, if $a \neq 0$, then $\bigcup_{n < \omega} \mathbb{Z}b_n$ is a non-trivial divisible subgroup of $G_2$, but $G_2$ is reduced, a contradiction.]
\newline We are then only left to prove $(\star_4)$. Fix $q \in \mathbb{P}$ and let $m = |K_q|$, then there is $b_0 \in G_2$ such that $q^m b_0 = a$, as $G_2 \models q^\infty \vert(a - c_{(a, q)})$ and by assumption $c_{(a, q)} = 0$. 
\begin{enumerate}[$(\star_{4.1})$]
	\item Let $b' = b_0 - c_{(b_0, q)}$.
\end{enumerate}
Then we have:
	\begin{enumerate}[$(\star_{4.2})$]
	\item 
	\begin{enumerate}[$(\cdot_1)$]
	\item $c_{(b', q)} = 0$;
	\item $q^mb' = q^mb_0 - q^m c_{(b_0, q)} = a - 0 = a$ (as $m = |K_q|$ and $c_{(b_0, q)} \in K_q$);
	\item if $q \neq p \in \mathbb{P}$, then $p^\infty \vert \, b'$.
	\end{enumerate}
\end{enumerate}
[Why $(\star_{4.2})$? $(\cdot_1)$ is because for every $n < \omega$, $q^n x = b_0 - c_{(b_0, q)} = b' = b' - 0$ and the sequence $(c_{(b', p)} : p \in \mathbb{P})$ is unique.  $(\cdot_2)$ is clear. Concerning $(\cdot_3)$, letting $0 < \ell < \omega$ and supposing that $a = p^\ell a_1$, we can find $i, j \in \mathbb{Z}$ such that $iq^m + jp^\ell = 1$, and so:
$$\begin{array}{rcl}
b' & = & (iq^m + jp^\ell) b' \\
  & = & iq^m b' + jp^\ell b' \\
  & = & ia + jp^\ell b' \\
  & = & i p^\ell a_1 + jp^\ell b' \\
  & = & p^\ell (i a_1) + p^\ell (jb') \\
   & = & p^\ell (i a_1 + jb'). \\
\end{array}$$
Hence, since $p^\infty \vert \, a$ (recalling $\bigwedge_{p \in \mathbb{P}} c_{(a, p)} = 0$), we have that $p^\infty \vert \, b'$, as wanted.] 
\newline Now, the $b'$ from $(\star_{4.2})$ is almost as wanted for $(\star_{4})$. To conclude it suffices to let $b = q^{m-1}b'$, which is easily seen to be as wanted. This concludes the proof of $(\star_{4})$.

\smallskip 
\noindent This suffices for $(\cdot_5)$. We now prove $(\cdot_6)$. For every prime $p$ and $0 < n < \omega$, $p^n \mid (a - c_{(a, p)})$, but, as $\pi \in \mrm{End}(G_2)$, we have that $p^n \mid (\pi(a) - \pi(c_{(a, p)}))$ and $\pi$ maps $\mrm{Tor}_p(G)$ into itself hence $\pi(c_{(a, p)}) \in \mrm{Tor}_p(G) = G_{(2, p)}$. Thus, the sequence $(c_{(\pi(a), p)} : p \in \mathbb{P})$ behaves as in $(\cdot_5)$ with respect to $\pi(a)$ and so we are done proving $(\cdot_6)$, as the sequence in $(\cdot_5)$ is unique. This concludes the proof of the claim.
	\end{proof}
	
	\begin{corollary}\label{hammer_corollary} In the context of Claim~\ref{i2} and letting $G_2$ and $K = \bigoplus \{K_{(2, p)} : p \in \mathbb{P} \}  = \bigoplus \{K_{p} : p \in \mathbb{P} \}$ be there, we have (with respect to the $\mathbb{Z}$-adic topology):
	\begin{enumerate}[(a)]
	\item the $\mathbb{Z}$-adic completion of $K$ is the group $\prod \{K_{p} : p \in \mathbb{P} \}$;
	\item $K$ is dense in $G_2$ (i.e., every element in $G_2$ is the limit of a Cauchy sequence), and so w.l.o.g. $G_2$ is a countable subgroup of $\prod_{p \in \mathbb{P}} K_{p}$ containing $K = \oplus_{p \in \mathbb{P}} K_{p}$;
	\item $(\sum \{c_{(a, p)} : p \in \mathbb{P}, p < n \}: n < \omega)$ (cf. \ref{i2}$(c)(\cdot_5)$) is Cauchy with limit $a$.
	\end{enumerate}
\end{corollary}

	\begin{proof}[Proof of Theorem~\ref{main_th6}] By Corollary \ref{hammer_corollary}.
	
	\end{proof}


\begin{thebibliography}{10}

\bibitem{kechris}
S. Adams and A. Kechris.
\newblock {\em Linear algebraic groups and countable Borel equivalence relations}.
\newblock J. Amer. Math. Soc. {\bf 13} (2000), no. 04, 909-943.

\bibitem{baer}
R. Baer.
\newblock {\em Groups without proper isomorphic quotient groups}.
\newblock Bull. Amer. Math. Soc. {\bf 50} (1944), 267-278.

\bibitem{b&p}
R. A. Beaumont and R. S. Pierce.
\newblock {\em Partly transitive modules and modules with proper isomorphic submodules}.
\newblock  Trans. Amer. Math. Soc. {\bf 91} (1959), 209-219.


\bibitem{calderoni}
F. Calderoni and S. Thomas.
\newblock {\em The bi-embeddability relation for countable abelian groups}.
\newblock Trans. Amer. Math. Soc. {\bf 371} (2019), no. 03, 2237-2254.

\bibitem{nilpotentn_book}
A. E. Clement, S. Majewicz, and M. Zyman.
\newblock {\em The theory of nilpotent groups}.
\newblock Birkhäuser/Springer, Cham, 2017.


\bibitem{entropy1}
D. Dikranian, B. Goldsmith, L. Salce and P. Zanardo.
\newblock {\em Algebraic entropy for Abelian groups}. 
\newblock Trans Amer. Math. Soc. {\bf 361} (2009), 3401-3434.

\bibitem{friedman_and_stanley}
H. Friedman and L. Stanley.
\newblock {\em A Borel reducibility theory for classes of countable structures}.
\newblock J. Symb. Log. {\bf 54} (1989), no. 03, 894-914.

\bibitem{fuch_vol1}
Laszlo Fuchs.
\newblock {\em Infinite abelian groups. Vol. I}.
\newblock Pure and Applied Mathematics, Vol. 36 Academic Press, New York-London 1970.


\bibitem{gao}
S. Gao.
\newblock {\em Invariant descriptive set theory}. 
\newblock Taylor \& Francis Inc, 2008.

\bibitem{entropy2}
B. Goldsmith and K. Gong.
\newblock {\em On adjoint entropy of Abelian groups}.
\newblock Comm. Algebra {\bf 40} (2012), no. 03, 972-987.

\bibitem{entropy3}
B. Goldsmith and K. Gong.
\newblock {\em Algebraic entropies, Hopficity and co-Hopficity of direct sums of Abelian groups}.
\newblock Topol. Algebra Appl. {\bf 03} (2015), 75-85.

\bibitem{hirshon}
R. Hirshon.
\newblock {\em Misbehaved direct products}.
\newblock Expo, Math. {\bf } (2002), 365.374.



\bibitem{1205}
G. Paolini and S. Shelah.
\newblock {\em Torsion-free abelian groups are Borel complete}. 
\newblock Ann. of Math. (2), accepted for publication (04.12.2023), link \underline{\href{https://annals.math.princeton.edu/articles/21268}{here}}.

\bibitem{1214}
G. Paolini and S. Shelah.
\newblock {\em On the existence of uncountable Hopfian and co-Hopfian abelian groups}. 
\newblock  Israel J. Math., accepted for publication.

\bibitem{pinsky}
B. Pinsky.
\newblock {\em Hopfian groups are complete co-analytic}. 
\newblock Available on the ArXiv.

\bibitem{decomposability}
K. Riggs.
\newblock {\em  The decomposability problem for torsion-free abelian groups is analytic-complete}.
\newblock Proc. Amer. Math. Soc. {\bf 143} (2015), no. 08, 3631-3640.


\bibitem{sh44}
S. Shelah.
\newblock {\em Infinite abelian groups, Whitehead problem and some constructions}. 
\newblock Israel J. Math. {\bf 18} (1974), 243-256.

\bibitem{thomas_complete_groups}
S. Thomas.
\newblock {\em Complete groups are complete co-analytic}. 
\newblock Arch. Math. Logic {\bf 57} (2018), no. 05-06, 601-606.

\bibitem{thomas_ACTA}
S. Thomas.
\newblock {\em On the complexity of the classification problem for torsion-free abelian groups of rank two}. 
\newblock Acta Math. {\bf 189} (2002), no. 02, 287-305.

\bibitem{thomas_JAMS}
S. Thomas.
\newblock {\em The classification problem for torsion-free abelian groups of finite rank}. 
\newblock J. Amer. Math. Soc. {\bf 16} (2003), no. 01, 233-258.


\end{thebibliography}
\end{document}